\documentclass[final,3p]{elsarticle}
\usepackage{lineno,hyperref}
\usepackage{amsmath}
\usepackage{amssymb}
\usepackage{amsthm}
\usepackage[linesnumbered,ruled,vlined]{algorithm2e}
\usepackage{caption}
\usepackage{mathtools}
\usepackage{xcolor}
\usepackage{changes}
\usepackage{multirow}
\usepackage{verbatim}
\usepackage{mathrsfs}
\usepackage{graphicx}
\usepackage{subcaption}
\usepackage{bm}
\usepackage{booktabs}
\allowdisplaybreaks[4]
\theoremstyle{plain}
\newtheorem{theorem}{Theorem}[section]

\theoremstyle{definition}

\theoremstyle{remark}
\newtheorem{remark}{Remark}

\modulolinenumbers[5]

\journal{ArXiv.org}

\begin{document}

\begin{frontmatter}

\title{G-RADI: A Low-Rank ADI Algorithm for General-Form Large-Scale Continuous-Time Algebraic Riccati Equations}

\author[uz]{Umair~Zulfiqar\corref{mycorrespondingauthor}}
\cortext[mycorrespondingauthor]{Corresponding author}
\ead{umair@yangtzeu.edu.cn}
\address[uz]{School of Electronic Information and Electrical Engineering, Yangtze University, Jingzhou, Hubei, 434023, China}
\begin{abstract}
Continuous-time algebraic Riccati equations (CAREs) arise in numerous disciplines and can be written in several forms. When the coefficient matrices are small and dense, stabilizing solutions can be computed for all these forms. The situation differs for large-scale sparse problems: for CAREs in general form, only a reformulation of the Newton--Kleinman iteration is currently available.

The low-rank alternating direction implicit (ADI) algorithm for CAREs, known as Riccati-ADI (RADI), is among the most successful and computationally efficient methods in the literature. However, RADI applies only to simplified forms of CAREs; no low-rank ADI algorithm is available for CAREs in general form. This paper presents a low-rank ADI algorithm for CAREs in general form, which contains RADI as a special case. Unlike the Newton--Kleinman iteration, which computes a new approximation at every step, the proposed algorithm---like RADI---recursively accumulates low-rank factors of an approximation to the stabilizing solution. Moreover, the residual admits a low-rank factorization, so its norm can be computed efficiently.

An automatic shift-generation strategy is also proposed, enabling the algorithm to select its shifts without user input. Numerical experiments on benchmark problems with orders ranging from $10^6$ to $10^7$ show that the proposed method is accurate and computationally efficient, and that it is faster than the Newton--Kleinman iteration.
\end{abstract}


\end{frontmatter}

\section{Introduction}
Solutions of continuous-time algebraic Riccati equations (CAREs) are central to systems and control theory. They arise in linear-quadratic regulator (LQR) and linear-quadratic-Gaussian (LQG) design \cite{zhoubook}, in $\mathcal{H}_2$ and $\mathcal{H}_\infty$ control \cite{mustafa1990controller}, and in passivity- and contractivity-preserving variants of balanced truncation \cite{green1988balanced,opdenacker1988contraction,jonckheere1983new}. This paper considers CAREs of the general form
\begin{align}
A^\top X E+ E^\top X A+ E^\top XB_2R_2^{-1}B_2^\top XE -(E^\top X B_1+C_2^\top)R_1^{-1}(B_1^\top X E+C_2)+C_1^\top Z C_1=0,\label{gen_ricc}
\end{align}
where $E,A\in\mathbb{R}^{n\times n}$, $B_1\in\mathbb{R}^{n\times m_1}$, $B_2\in\mathbb{R}^{n\times m_2}$, $C_1\in\mathbb{R}^{p\times n}$, $C_2\in\mathbb{R}^{m_1\times n}$, while $X\in\mathbb{R}^{n\times n}$, $R_1\in\mathbb{R}^{m_1\times m_1}$, $R_2\in\mathbb{R}^{m_2\times m_2}$, and $Z\in\mathbb{R}^{p\times p}$ are symmetric. It is assumed that $E$, $R_1$, and $R_2$ are invertible; the matrices $R_1$, $R_2$, and $Z$ may be indefinite. Equation \eqref{gen_ricc} is the form solved by MATLAB's \texttt{icare} command and includes all CARE formulations considered in earlier works \cite{lin2015new,benner2018radi,bertram2024family,benner2023low,saak2024using,zulfiqar2025unified}.

Define
\begin{align}
\hat{A}=A+\tilde{B}\hat{Z}\hat{C},\qquad
\tilde{B}=\begin{bmatrix}0 & B_1\end{bmatrix},\qquad
\hat{B}=\begin{bmatrix}B_1 & B_2\end{bmatrix},\nonumber\\
\hat{R}=\begin{bmatrix}R_1 & 0 \\ 0 & -R_2\end{bmatrix},\qquad
\hat{C}=\begin{bmatrix}C_1 \\ C_2\end{bmatrix},\qquad
\hat{Z}=\begin{bmatrix}Z & 0 \\ 0 & -R_1^{-1}\end{bmatrix}.\nonumber
\end{align}
Then \eqref{gen_ricc} can be rewritten as
\begin{align}
\hat{A}^\top XE + E^\top X \hat{A} - E^\top X \hat{B} \hat{R}^{-1} \hat{B}^\top X E + \hat{C}^\top \hat{Z} \hat{C} = 0.\label{care}
\end{align}
Define the gain matrix and closed-loop state matrix by
\begin{align}
K_{\mathrm{gain}} &= R_1^{-1}\big(B_1^\top XE + C_2\big),\nonumber\\
A_{\mathrm{cl}} &= \Big(A + B_2 R_2^{-1} B_2^\top XE - B_1 K_{\mathrm{gain}}\Big)E^{-1}.\label{K_gain}
\end{align}
Among the solutions of \eqref{gen_ricc}, the stabilizing solution is sought, characterized by the Hurwitz property of the closed-loop matrix $A_{\mathrm{cl}}$; that is, all eigenvalues of $A_{\mathrm{cl}}$ have negative real parts.

For small, dense problems, the stabilizing solution of \eqref{gen_ricc} can be computed in essentially all cases, irrespective of the definiteness of $R_1$, $R_2$, and $Z$ and of whether $B_2\neq0$ or $C_2\neq0$. Direct methods obtain it from an invariant subspace of the associated Hamiltonian matrix or matrix pencil \cite{laub1979schur}, whereas iterative methods, including those based on the matrix sign function \cite{roberts1980linear,benner2014factored}, avoid computing an eigendecomposition. Both classes of methods require dense $n\times n$ computations and produce a dense solution, and therefore become impractical for large $n$.

If $p,m_1,m_2\ll n$, the stabilizing solution is usually of low numerical rank. Low-rank approximation methods can exploit this structure and remain efficient for large-scale problems with sparse matrices $A$ and $E$. For coefficient matrices arising, for example, from discretized partial differential equations, the solution remains dense but has rapidly decaying singular values when $p,m_1,m_2\ll n$ \cite{benner2016solution}. Low-rank iterative solvers have been proposed for several special cases of \eqref{gen_ricc}. In the classical setting, where $R_1\geq0$, $B_2=0$, $Z\geq0$, and $C_2=0$, these methods include low-rank Newton--Kleinman iterations with inner low-rank ADI solves for Lyapunov equations \cite{benner2008numerical,benner2016inexact,saak2009efficient}, projection methods that solve reduced dense CAREs on approximating subspaces \cite{heyouni2009extended,simoncini2016analysis}, invariant-subspace and subspace-iteration approaches \cite{amodei2010invariant,benner2016solution,lin2015new}, and ADI-type Riccati solvers. The latter are derived from the quadratic ADI iteration \cite{wong2005quadratic,wong2007fast} and include the low-rank Riccati-ADI (RADI) method \cite{benner2018radi}, which builds the solution factors incrementally by appending a small number of columns in each iteration. A family of related low-rank algorithms is analyzed in \cite{bertram2024family}. Overviews and numerical comparisons of these solvers are given in \cite{benner2013numerical,benner2020numerical}.

When $R_1$, $R_2$, and $Z$ are not positive semidefinite and both $B_2$ and $C_2$ are nonzero, the number of available methods decreases sharply. For the case $R_2\geq0$, $B_2\neq0$, $Z\geq0$, $C_2=0$, and $B_1=0$, Newton and Newton--Kleinman iterations have been extended in \cite{benner2014projected}. Recently, a low-rank method was proposed in \cite{benner2023low} for CAREs with indefinite $R_1$ and symmetric positive semidefinite $Z$. It was observed in \cite{zulfiqar2025unified} that, with a slight modification, RADI can be generalized to handle the same case.

For the fully general form \eqref{gen_ricc}, the Newton--Kleinman iteration with $LDL^\top$ factorizations \cite{saak2024using} is the only available large-scale low-rank solver. Its cost is dominated by a Lyapunov solve in every outer iteration. This solve can be approximated using the $LDL^\top$ variant of the low-rank ADI algorithm for Lyapunov equations \cite{lang2015benefits}. Because the low-rank solution factors are recomputed at every Newton--Kleinman iteration rather than accumulated recursively, this approach is less efficient than RADI \cite{benner2018radi}. Consequently, generalizing RADI to the general form \eqref{gen_ricc} has been identified as a direction for future work.

In this paper, RADI is generalized to the general-form CARE \eqref{gen_ricc} with indefinite $Z$, $R_1$, and $R_2$, and with $B_2\neq0$ and $C_2\neq0$. As in the original RADI method \cite{benner2018radi}, the proposed algorithm accumulates the low-rank solution factors incrementally. Consequently, its cost is comparable to that of a single inner low-rank ADI sweep in the Newton--Kleinman iteration \cite{saak2024using}.

A further advantage of the proposed algorithm is that the residual admits an expression of low, fixed rank. Hence, its spectral or Frobenius norm can be computed inexpensively. A Petrov--Galerkin projection-based interpretation of the proposed algorithm is also provided, showing that this algorithm implicitly computes a stabilizing solution of a projected CARE. Furthermore, it is shown that the algorithm implicitly performs recursive interpolation-based model order reduction (MOR). The resulting reduced-order model (ROM), which is not explicitly formed but can be extracted, interpolates the state-space model defining \eqref{gen_ricc} at the mirror images of the ADI shifts. Based on this observation, an automatic shift-generation strategy is proposed that selects the dominant Ritz values of $AE^{-1}$ as ADI shifts. Numerical experiments are performed on one moderate-scale example and two large-scale problems. For the moderate-scale example, the proposed algorithm is compared with MATLAB's \texttt{icare} command and the Newton--Kleinman iteration. For the large-scale examples, it is compared with the Newton--Kleinman iteration, since MATLAB's \texttt{icare} cannot be used. The results confirm that the proposed algorithm is computationally efficient and, with the proposed automatic shift generation, achieves excellent accuracy.

The remainder of the paper is organized as follows. Section \ref{sec:existing} recalls the Newton--Kleinman iteration for \eqref{gen_ricc}. Section \ref{sec:main} develops the generalized RADI method, establishes its properties, discusses computational aspects, and presents its projection and MOR interpretations together with the shift-generation strategy. Section \ref{sec:numerics} reports the numerical results, and Section \ref{sec:conclusion} concludes the paper. Proofs of the two theorems are given in Appendices A and B.

Throughout the paper, $\mathbb{R}^{n\times m}$ denotes the set of real $n\times m$ matrices, and $\mathbb{C}_{-}$ denotes the open left half of the complex plane. For a matrix $M$, $M^\top$ denotes its transpose, $\lambda(M)$ its spectrum, and $\|M\|_2$ and $\|M\|_F$ its spectral and Frobenius norms. The notation $M>0$ ($M\geq0$) indicates that $M$ is symmetric positive (semi)definite, while a symmetric matrix that is neither positive nor negative semidefinite is called indefinite. The real and imaginary parts are denoted by $\mathrm{Re}(\cdot)$ and $\mathrm{Im}(\cdot)$, the identity matrix of order $k$ by $I_k$, a block diagonal matrix by $\mathrm{blkdiag}(\cdot)$, an orthonormal basis of the range by $\mathrm{orth}(\cdot)$, and an empty matrix by $[\;]$. A matrix is Hurwitz if all its eigenvalues lie in $\mathbb{C}_{-}$.
\section{Existing Algorithm}\label{sec:existing}
The Newton--Kleinman iteration for the CARE \eqref{gen_ricc}, as formulated in \cite{saak2024using}, is recalled here because it is both the method against which the proposed algorithm is compared and the only low-rank solver currently available for the general form \eqref{gen_ricc}. The iteration is stated in Algorithm \ref{alg_nl}. Each outer step requires solving the Lyapunov equation \eqref{lyap_nl}. In the large-scale sparse setting, this equation is itself solved iteratively, and the $LDL^\top$ variant of the low-rank ADI algorithm \cite{lang2015benefits} is used to obtain a factored approximation $\hat{X}_k\approx L_kD_kL_k^\top$; this is the inexact Newton--Kleinman iteration of \cite{saak2024using}. The Lyapunov solve is restarted at every outer step, so the factors accumulated in previous steps are not reused. It is precisely this repeated inner solve that the algorithm developed in the next section avoids.
\begin{algorithm}
\caption{Newton--Kleinman Iteration}\label{alg_nl}
\DontPrintSemicolon
\KwIn{
  Matrices of CARE \eqref{gen_ricc}: $E$, $A$, $B_1$, $B_2$, $R_1$, $R_2$, $C_1$, $Z$, $C_2$;  Tolerance: $0<\tau<1$; Initial Stabilizing Gain: $\hat{K}_0$.
}
\KwOut{
  Approximation of $X$: $X \approx\hat{X}_k= L_kD_kL_k^\top$; Approximation of gain matrix $K_{\mathrm{gain}}$: $K_{\mathrm{gain}}\approx \tilde{K}_k=R_1^{-1}\big(B_1^\top \hat{X}_k E+C_2\big)$.
}

\BlankLine
\textbf{Initialization:} $\hat{A}=A-B_1R_1^{-1}C_2$, $\hat{B} = \begin{bmatrix}B_1 & B_2\end{bmatrix}$, $\hat{R} = \mathrm{blkdiag}(R_1, -R_2)$, $\hat{C} = \begin{bmatrix}C_1 \\ C_2\end{bmatrix}$, $\hat{Z} = \mathrm{blkdiag}(Z, -R_1^{-1})$, $\hat{T} = \mathrm{blkdiag}(\hat{Z},\hat{R})$, $k = 1$.

\BlankLine
\While{$\dfrac{\big\|\hat{A}^\top \hat{X}_kE + E^\top \hat{X}_k \hat{A} - E^\top \hat{X}_k \hat{B} \hat{R}^{-1} \hat{B}^\top \hat{X}_k E + \hat{C}^\top \hat{Z} \hat{C}\big\|}{\big\|\hat{C}^\top\hat{Z}\hat{C}\big\|} \geq \tau$}{
  Set $\hat{W}_k=\begin{bmatrix}\hat{C}\\\hat{K}_{k-1}\end{bmatrix}$ and solve the following Lyapunov equation:
  \begin{align}
  \big(\hat{A}-\hat{B}\hat{K}_{k-1}\big)^\top \big(L_kD_kL_k^\top\big) E + E^\top \big(L_kD_kL_k^\top\big) \big(\hat{A}-\hat{B}\hat{K}_{k-1}\big)+(\hat{W}_k)^\top \hat{T}\hat{W}_k=0.\label{lyap_nl}
  \end{align}
  Update $\hat{K}_k=\hat{R}^{-1}(\hat{B}^\top L_kD_kL_k^\top E)$ and $k=k+1$.
    }
    Set $\tilde{K}_k=R_1^{-1}\big(B_1^\top L_kD_kL_k^\top E+C_2\big)$.
\end{algorithm}

The residual \(R(\hat{X}_k)\) in the Newton–Kleinman iteration is given by
\[
R(\hat{X}_k)=\hat{A}^\top \hat{X}_kE + E^\top \hat{X}_k \hat{A} - E^\top \hat{X}_k \hat{B} \hat{R}^{-1} \hat{B}^\top \hat{X}_k E + \hat{C}^\top \hat{Z} \hat{C}=-\big(\hat{K}_k-\hat{K}_{k-1}\big)^\top \hat{R}\big(\hat{K}_k-\hat{K}_{k-1}\big).
\]
If the Lyapunov equation \eqref{lyap_nl} is solved inexactly, the residual \(R(\hat{X}_k)\) becomes
\[
R(\hat{X}_k)=R_{\mathrm{lyap}}-\big(\hat{K}_k-\hat{K}_{k-1}\big)^\top \hat{R}\big(\hat{K}_k-\hat{K}_{k-1}\big),
\]
where
\[
 \big(\hat{A}-\hat{B}\hat{K}_{k-1}\big)^\top \big(L_kD_kL_k^\top\big) E + E^\top \big(L_kD_kL_k^\top\big) \big(\hat{A}-\hat{B}\hat{K}_{k-1}\big)+(\hat{W}_k)^\top \hat{T}\hat{W}_k=R_{\mathrm{lyap}}.
\]
Thus, if \(R_{\mathrm{lyap}}\) admits a low-rank factorization—as is the case for the \(LDL^\top\) variant of the low-rank ADI algorithm \cite{lang2015benefits}—then the residual \(R(\hat{X}_k)\) also admits a low-rank factorization, which makes the computation of the spectral or Frobenius norm of \(R(\hat{X}_k)\) computationally efficient even in large-scale settings.
\section{Main Work}\label{sec:main}
In this section, a generalized low-rank ADI algorithm for the CARE \eqref{gen_ricc}/\eqref{care} is developed, and its theoretical properties, efficient implementation, and automatic ADI shift generation are discussed.
\subsection{The Low-Rank ADI Solver for General-Form CAREs}
Let \(\{\alpha_i\}_{i=1}^{k} \subset \mathbb{C}_{-}\) be a sequence of ADI shifts such that $-\alpha_i$ is not an eigenvalue of $\hat{A}E^{-1}$. Complex-valued shifts are assumed to occur in conjugate pairs: if $\mathrm{Im}(\alpha_i)\neq0$, then $\alpha_{i+1}=\overline{\alpha_i}$. Each conjugate pair is listed once and processed in a single iteration, ensuring that all quantities below remain real. Based on the shifts $\alpha_i$, define $s_r^{(i)}$ and $l_r^{(i)}$ as
\begin{align}
s_r^{(i)} & = 
\begin{cases} 
-\alpha_i I_{pm} & \text{if } \mathrm{Im}(\alpha_i) = 0, \\
\begin{bmatrix} 
-\mathrm{Re}(\alpha_i)I_{pm} & -\mathrm{Im}(\alpha_i)I_{pm} \\ 
\mathrm{Im}(\alpha_i)I_{pm} & -\mathrm{Re}(\alpha_i) I_{pm}
\end{bmatrix} & \text{if } \mathrm{Im}(\alpha_i) \neq 0,
\end{cases}\label{sw}\\
l_r^{(i)} &= 
\begin{cases} 
-I_{pm} & \text{if } \mathrm{Im}(\alpha_i) = 0, \\
\begin{bmatrix} 
-I_{pm} & 0
\end{bmatrix} & \text{if } \mathrm{Im}(\alpha_i) \neq 0,
\end{cases}\label{lr}
\end{align}
where \(I_{pm}\) denotes the identity matrix of size \((p+m_1) \times (p+m_1)\).

The following matrices are defined recursively:
\begin{align}
\hat{A}_r^{(i)}&=(S_r^{(i)})^\top - (L_r^{(i)})^\top \hat{Z} \hat{C}_r^{(i)}=\begin{bmatrix}
(S_r^{(i-1)})^\top - (L_r^{(i-1)})^\top \hat{Z} \hat{C}_r^{(i-1)}&-(L_r^{(i-1)})^\top\hat Z\hat c_r^{(i)}\\
\hat b_r^{(i)}\hat R^{-1}(\hat B_r^{(i-1)})^\top X_r^{(i-1)}&(s_r^{(i)})^\top-(l_r^{(i)})^\top\hat Z\hat c_r^{(i)}
\end{bmatrix}\nonumber\\
&=\begin{bmatrix}
\hat{A}_r^{(i-1)}&-(L_r^{(i-1)})^\top\hat Z\hat c_r^{(i)}\\
\hat b_r^{(i)}\hat R^{-1}(\hat B_r^{(i-1)})^\top X_r^{(i-1)}&a_r^{(i)}
\end{bmatrix},\label{a_recc}\\
\hat{B}_r^{(i)}&=(W^{(i)})^\top\hat{B}=\begin{bmatrix}(W^{(i-1)})^\top \hat{B}\\
w_i^\top \hat{B}\end{bmatrix}=\begin{bmatrix}\hat{B}_r^{(i-1)}\\\hat{b}_r^{(i)}\end{bmatrix},\\
\hat{C}_r^{(i)}&=L_r^{(i)}X_r^{(i)}=\begin{bmatrix}L_r^{(i-1)}X_r^{(i-1)}&l_r^{(i)}x_r^{(i)}\end{bmatrix}=\begin{bmatrix}\hat{C}_r^{(i-1)}&\hat{c}_r^{(i)}\end{bmatrix},\\
S_r^{(i)}&=\begin{bmatrix}S_r^{(i-1)}&(\hat{C}_r^{(i-1)})^\top\hat{Z}l_r^{(i)}+X_r^{(i-1)}\hat{B}_r^{(i-1)}\hat{R}^{-1}(\hat{b}_r^{(i)})^\top\\0&s_r^{(i)}\end{bmatrix},\\
L_r^{(i)}&=\begin{bmatrix}L_r^{(i-1)}&l_r^{(i)}\end{bmatrix},\\
W^{(i)}&=\begin{bmatrix}W^{(i-1)}&w_i\end{bmatrix},\\
C_{\perp}^{(i)}&=\hat{C}-\hat{C}_r^{(i)}(W^{(i)})^\top E=C_{\perp}^{(i-1)}- \hat{c}_r^{(i)}w_i^\top E,\\
X_r^{(i)}&=\begin{bmatrix}X_r^{(i-1)}&0\\0&x_r^{(i)}\end{bmatrix},\\
\hat{K}^{(i)}&=\hat{R}^{-1}\hat{B}^\top W^{(i)}X_r^{(i)}(W^{(i)})^\top E=\hat{K}^{(i-1)}+\hat{R}^{-1}(\hat{b}_r^{(i)})^\top x_r^{(i)}w_i^\top  E.\label{K_recc}
\end{align}
The initial values are $\hat{A}_r^{(0)}=[\;]$, $\hat{B}_r^{(0)}=[\;]$, $\hat{C}_r^{(0)}=[\;]$, $S_r^{(0)}=[\;]$, $L_r^{(0)}=[\;]$, $W^{(0)}=[\;]$, $C_{\perp}^{(0)}=\hat{C}$, $X_r^{(0)}=[\;]$, and $\hat{K}^{(0)}=0$.

The new blocks $(x_r^{(i)})^{-1}$ and $w_i$ are computed from the following linear matrix equations:
\begin{align}
-(s_r^{(i)})^\top (x_r^{(i)})^{-1}-(x_r^{(i)})^{-1}s_r^{(i)}+(l_r^{(i)})^\top\hat{Z}l_r^{(i)}+\hat{b}_r^{(i)}\hat{R}^{-1}(\hat{b}_r^{(i)})^\top=0,\label{x_lyap}\\
\big(\hat{A}^\top- (\hat{K}^{(i-1)})^\top \hat{B}^\top\big)w_i-E^\top w_is_r^{(i)}+(C_{\perp}^{(i-1)})^\top\hat{Z}l_r^{(i)}=0.\label{wi}
\end{align}
Equation \eqref{wi} is a Sylvester equation for $w_i$. Define
\[
\bar{w}_i=\Big(\hat{A}^\top - (\hat{K}^{(i-1)})^\top \hat{B}^\top +\alpha_i E^\top \Big)^{-1}\big(C_{\perp}^{(i-1)}\big)^\top.
\]
Its solution can then be expressed in real-valued form as
\begin{align}
w_i &=
\begin{cases}
\bar{w}_i\hat{Z} & \text{if } \mathrm{Im}(\alpha_i) = 0, \\
\begin{bmatrix} \mathrm{Re}(\bar{w}_i\hat{Z}) & \mathrm{Im}(\bar{w}_i\hat{Z}) \end{bmatrix}
=\begin{bmatrix} w_i^{R} & w_i^{I} \end{bmatrix} & \text{if } \mathrm{Im}(\alpha_i) \neq 0.
\end{cases}\label{w_i}
\end{align}
The block $(x_r^{(i)})^{-1}$ is given explicitly by
\begin{align}
(x_r^{(i)})^{-1} &=
\begin{cases}
-\frac{1}{2\,\mathrm{Re}(\alpha_i)}\big( \hat{Z} + \hat{b}_r^{(i)}\hat{R}^{-1}(\hat{b}_r^{(i)})^\top \big) & \text{if } \mathrm{Im}(\alpha_i) = 0, \\[6pt]
\begin{bmatrix} p_{11} & p_{12} \\ p_{12}^\top & p_{22} \end{bmatrix} & \text{if } \mathrm{Im}(\alpha_i) \neq 0,
\end{cases}\label{x_r}
\end{align}
where
\begin{align}
p_{11} &= -\frac{1}{\Delta} \Big( \gamma_1 \big( \hat{Z} + q_{11} \big) + \gamma_2 q_{22} + \gamma_3 \big( q_{12} + q_{12}^\top \big) \Big), \\[4pt]
p_{12} &= \frac{1}{\Delta} \Big( \gamma_3 \big( \hat{Z} + q_{11} - q_{22} \big) - \gamma_1 q_{12} + \gamma_2 q_{12}^\top \Big), \\[4pt]
p_{22} &= \frac{1}{\Delta} \Big( \gamma_3 \big( q_{12} + q_{12}^\top \big) - \gamma_2 \big( \hat{Z} + q_{11} \big) - \gamma_1 q_{22} \Big), \\[4pt]
\gamma_1 &= 2\,\mathrm{Re}(\alpha_i)^2 + \mathrm{Im}(\alpha_i)^2, \quad
\gamma_2 = \mathrm{Im}(\alpha_i)^2, \quad
\gamma_3 = \mathrm{Re}(\alpha_i)\,\mathrm{Im}(\alpha_i), \\[4pt]
\Delta &= 4\,\mathrm{Re}(\alpha_i) \Big( \big(\mathrm{Re}(\alpha_i)\big)^2 + \big(\mathrm{Im}(\alpha_i)\big)^2 \Big), \\[4pt]
q_{11} &= (w_i^{R})^\top \hat{B} \hat{R}^{-1} \hat{B}^\top w_i^{R}, \quad
q_{12} = (w_i^{R})^\top \hat{B} \hat{R}^{-1} \hat{B}^\top w_i^{I}, \quad
q_{22} = (w_i^{I})^\top \hat{B} \hat{R}^{-1} \hat{B}^\top w_i^{I}.\label{x_complex2}
\end{align}
Substitution of the recursive formulas \eqref{a_recc}--\eqref{K_recc} shows that $W^{(i)}$ and $(X_r^{(i)})^{-1}$ satisfy
\begin{align}
\hat{A}^\top W^{(i)}-E^\top W^{(i)} S_r^{(i)}+\hat{C}^\top\hat{Z}L_r^{(i)}&=0,\label{radi_sylv}\\
\hat{A}^\top W^{(i)} - E^\top W^{(i)} (\hat{A}_r^{(i)})^\top + (C_{\perp}^{(i)})^\top \hat{Z} L_r^{(i)} &= 0,\label{radi_sylv2}\\
-(S_r^{(i)})^\top (X_r^{(i)})^{-1}-(X_r^{(i)})^{-1}S_r^{(i)}+(L_r^{(i)})^\top\hat{Z}L_r^{(i)}+\hat{B}_r^{(i)}\hat{R}^{-1}(\hat{B}_r^{(i)})^\top&=0.\label{radi_lyap}
\end{align}
Under the standing assumption on the shifts, the spectra of $S_r^{(i)}$ and $\hat{A}E^{-1}$ are disjoint, so the Sylvester equation \eqref{radi_sylv} has a unique solution. The following theorem establishes the theoretical properties of the approximation $X\approx W^{(i)}X_r^{(i)}(W^{(i)})^\top$.
\begin{theorem}\label{th1}
Let \(\{\alpha_i\}_{i=1}^{k}\subset\mathbb{C}_{-}\) be ADI shifts such that $-\alpha_i$ is not an eigenvalue of $\hat{A}E^{-1}$. Let \(W^{(i)}\) solve the Sylvester equation \eqref{radi_sylv}, and assume that the recursive formulas \eqref{a_recc}--\eqref{K_recc} are well defined. Then the following statements hold for $i=1,\dots,k$:
\begin{enumerate}
  \item \(X_r^{(i)}\) is a stabilizing solution of the reduced CARE
  \begin{align}
(\hat{A}_r^{(i)})^\top X_r^{(i)}+ X_r^{(i)} \hat{A}_r^{(i)}- X_r^{(i)}\hat{B}_r^{(i)}\hat{R}^{-1}(\hat{B}_r^{(i)})^\top X_r^{(i)}+(\hat{C}_r^{(i)})^\top \hat{Z}\hat{C}_r^{(i)}&=0.\label{proj_ricc}
  \end{align}
  The reduced closed-loop state matrix
  \[
  \hat{A}_{\mathrm{cl}}^{(i)} = \hat{A}_r^{(i)} - \hat{B}_r^{(i)} \hat{R}^{-1} (\hat{B}_r^{(i)})^\top X_r^{(i)}
  \]
  has eigenvalues \(\alpha_1,\dots,\alpha_i\), each with multiplicity \(p+m_1\).

  \item The residual $R(\hat{X}^{(i)})$ associated with the low-rank approximation $\hat{X}^{(i)}=W^{(i)}X_r^{(i)}(W^{(i)})^\top$ satisfies
  \[
  R(\hat{X}^{(i)})=\hat{A}^\top \hat{X}^{(i)}E + E^\top \hat{X}^{(i)} \hat{A} - E^\top \hat{X}^{(i)} \hat{B} \hat{R}^{-1} \hat{B}^\top \hat{X}^{(i)} E + \hat{C}^\top \hat{Z} \hat{C} = \big(C_{\perp}^{(i)}\big)^\top\hat{Z} C_{\perp}^{(i)}.
  \]
  Thus, the residual has rank at most $p+m_1$ and is available in factored form through the recursively updated matrix $C_{\perp}^{(i)}$.

  \item The gain matrix \(K_{\mathrm{gain}}\) is approximated by the recursion
  \[
  K_{\mathrm{gain}} \approx \tilde{K}_{\mathrm{gain}}^{(i)}
  =\tilde{K}_{\mathrm{gain}}^{(i-1)}
  + R_1^{-1} B_1^\top w_i x_r^{(i)} w_i^\top E,
  \]
  initialized with $\tilde{K}_{\mathrm{gain}}^{(0)}=R_1^{-1}C_2$.
\end{enumerate}
\end{theorem}
\begin{proof}
The proof is given in Appendix A.
\end{proof}

The resulting algorithm, referred to as Generalized RADI (G-RADI), is presented in Algorithm \ref{alg1}.
\begin{algorithm}
\caption{G-RADI}\label{alg1}
\DontPrintSemicolon
\KwIn{
  Matrices of CARE \eqref{gen_ricc}: $E$, $A$, $B_1$, $B_2$, $R_1$, $R_2$, $C_1$, $Z$, $C_2$; ADI shifts: $\{\alpha_i\}_{i=1}^k \subset \mathbb{C}_{-}$; Tolerance: $0<\tau<1$.
}
\KwOut{
  Approximation of $X$: $X \approx\hat{X}^{(i)}= W^{(i)}X_r^{(i)}(W^{(i)})^\top$; Approximation of gain matrix $K_{\mathrm{gain}}$: $K_{\mathrm{gain}}\approx \tilde{K}_{\mathrm{gain}}^{(i)}=R_1^{-1}\big(B_1^\top \hat{X}^{(i)} E+C_2\big)$, Residual: $R(\hat{X}^{(i)}) = \big(C_{\perp}^{(i)}\big)^\top\hat{Z} C_{\perp}^{(i)}$.
}

\BlankLine
\textbf{Initialization:} $\hat{B} = \begin{bmatrix}B_1 & B_2\end{bmatrix}$, $\hat{R} = \mathrm{blkdiag}(R_1, -R_2)$, $\hat{C} = \begin{bmatrix}C_1 \\ C_2\end{bmatrix}$, $\hat{Z} = \mathrm{blkdiag}(Z, -R_1^{-1})$, $C_{\perp}^{(0)} = \hat{C}$, $W^{(0)} = [\;]$, $X_r^{(0)} = [\;]$, $\hat{K}^{(0)} = 0$, $\tilde{K}_\mathrm{gain}^{(0)}=R_1^{-1}C_2$, $i = 1$.

\BlankLine
\While{$\dfrac{\big\|\big(C_{\perp}^{(i-1)}\big)^\top\hat{Z} C_{\perp}^{(i-1)}\big\|}{\big\|\hat{C}^\top\hat{Z}\hat{C}\big\|} \geq \tau$}{
  Solve for $\bar{w}_i$: $\big(A^\top-C_2^\top R_1^{-1}B_1^\top - (\hat{K}^{(i-1)})^\top \hat{B}^\top +\alpha_i E^\top \big) \bar{w}_i = \big(C_{\perp}^{(i-1)}\big)^\top$.\label{step1}\\
   
    Set $w_i$ as in \eqref{w_i} and expand $W^{(i)} = \begin{bmatrix} W^{(i-1)} & w_i \end{bmatrix}$. \\
    Compute $(x_r^{(i)})^{-1}$ from \eqref{x_r}--\eqref{x_complex2} and expand $X_r^{(i)} = \mathrm{blkdiag}\big(X_r^{(i-1)}, x_r^{(i)}\big)$. \\
    Set $l_r^{(i)}$ as in \eqref{lr} and $\hat{c}_r^{(i)}=l_r^{(i)}x_r^{(i)}$.\\
    Update $C_{\perp}^{(i)}\gets C_{\perp}^{(i-1)}- \hat{c}_r^{(i)}w_i^\top E$, $\hat{K}^{(i)} \gets \hat{K}^{(i-1)} + \hat{R}^{-1}\big(\hat{B}^\top w_i x_r^{(i)} w_i^\top E\big)$, $\tilde{K}_\mathrm{gain}^{(i)}\gets\tilde{K}_\mathrm{gain}^{(i-1)}+R_1^{-1}\big(B_1^\top w_i x_r^{(i)} w_i^\top E\big)$, and $i\gets i+1$.
    }
\end{algorithm}
\begin{remark}
G-RADI includes several existing low-rank ADI algorithms as special cases:
\begin{enumerate}
  \item When \(B_2=0\), \(C_2=0\), \(R_1=I\), and \(Z=I\), G-RADI reduces to the original RADI algorithm \cite{benner2018radi}.
  \item When \(\hat{B}=0\), G-RADI reduces to the \(LDL^\top\) variant of the low-rank ADI algorithm for Lyapunov equations \cite{lang2015benefits}.
  \item When \(\hat{B}=0\) and \(\hat{Z}=I\), G-RADI reduces to the standard low-rank ADI algorithm for Lyapunov equations \cite{li2002low}.
\end{enumerate}
\end{remark}
\begin{remark}
The packed CARE \eqref{care} is a standard CARE with data
\((\hat{A},\hat{B},\hat{C},\hat{R},\mathcal{Q})\), where
\(\mathcal{Q}=\hat{C}^\top\hat{Z}\hat{C}\). If
\(\hat{Z}\geq0\) and
\(\mathcal{G}=\hat{B}\hat{R}^{-1}\hat{B}^\top\geq0\), it is of
LQR type. In this case, the convergence results for RADI
\cite{benner2018radi} yield
\[
0\leq\hat{X}^{(1)}\leq\hat{X}^{(2)}\leq\cdots\leq X.
\]
Furthermore, if \(\hat{A}E^{-1}\) is stable, then
\(\hat{X}^{(i)}\to X\) under the non-Blaschke condition $\sum_i\frac{\operatorname{Re}(\alpha_i)}
{1+|\alpha_i|^2}=-\infty$; cf. \cite{benner2018radi}. In terms of the original matrices, these positivity conditions
hold, for example, if \(C_2=0\), \(Z\geq0\), \(R_1>0\),
and either \(B_2=0\) or \(R_2<0\). This includes the classical
LQR/LQG setting. The conditions also hold if \(R_1<0\),
\(R_2<0\), \(Z\geq0\), and $B_1R_1^{-1}B_1^\top-B_2R_2^{-1}B_2^\top\geq0$.
For a fully indefinite instance of CARE \eqref{gen_ricc}, however,
the small blocks \(x_r^{(i)}\) need not be positive semidefinite.
Thus, monotonicity of the iterates cannot be guaranteed in general.
\end{remark}
\subsection{Computational Aspects}
Several aspects of Algorithm \ref{alg1} require attention for an efficient implementation. Convergence can be monitored using the residual factorization
\[
R(\hat{X}^{(i)})=\big(C_{\perp}^{(i)}\big)^\top\hat{Z}C_{\perp}^{(i)}.
\]
The residual has rank at most $p+m_1$, allowing its spectral or Frobenius norm to be computed from a matrix of order $p+m_1$ rather than $n$. Since $p+m_1\ll n$, this computation is inexpensive. The dominant cost is the shifted solve
\[
\bar{w}_i = \big(A^\top -C_2^\top R_1^{-1}B_1^\top - (\hat{K}^{(i-1)})^\top \hat{B}^\top +\alpha_i E^\top \big)^{-1}\big(C_{\perp}^{(i-1)}\big)^\top.
\]
The coefficient matrix is a low-rank modification of $A^\top+\alpha_i E^\top$, so the Sherman--Morrison--Woodbury (SMW) formula \cite{golub2013matrix} can be used. The system is written as
\[
\big(A^\top+\mathscr{U}\mathscr{V}^\top+\alpha_i E^\top\big)\bar{w}_i
=\big(C_{\perp}^{(i-1)}\big)^\top,
\]
where
\[
\mathscr{U}=\begin{bmatrix}
-C_2^\top R_1^{-1} & -(\hat K^{(i-1)})^\top
\end{bmatrix},
\qquad
\mathscr{V}=\begin{bmatrix}
B_1 & \hat B
\end{bmatrix}.
\]
First, $\mathscr{Y}$ and $\mathscr{Z}$ are computed by solving
\[
(A^\top+\alpha_i E^\top)
\begin{bmatrix}
\mathscr{Y} & \mathscr{Z}
\end{bmatrix}
=
\begin{bmatrix}
(C_{\perp}^{(i-1)})^\top & \mathscr{U}
\end{bmatrix}.
\]
The small system
\[
(I+\mathscr{V}^\top\mathscr{Z})\mathscr{X}
=\mathscr{V}^\top\mathscr{Y}
\]
is then solved for $\mathscr{X}$, giving
\[
\bar{w}_i=\mathscr{Y}-\mathscr{Z}\mathscr{X}.
\]
Computing $\mathscr{X}$ is inexpensive because the matrices involved are small. The main computational cost therefore lies in the shifted linear solve for $\begin{bmatrix}\mathscr{Y}&\mathscr{Z}\end{bmatrix}$, whose right-hand side has $p+3m_1+m_2$ columns. The small number of right-hand sides makes this approach suitable for large-scale sparse problems.

It is now shown that $\bar{w}_i$ can be computed using shifted linear solves with only $p+m_1$ right-hand sides instead of $p+3m_1+m_2$, providing a more efficient alternative to the SMW approach. First, recall the Cholesky-factor ADI (CF-ADI) iteration \cite{benner2013reformulated}. Define
\begin{align}
s_{\mathrm{lyap}}^{(i)} &=
\begin{cases}
-\alpha_i I_{pm} & \text{if } \mathrm{Im}(\alpha_i)=0, \\[6pt]
\gamma_i^2
\begin{bmatrix}
I_{pm} & \dfrac{\sqrt{1+\delta_i^2}}{2\delta_i}I_{pm} \\[6pt]
-\dfrac{\sqrt{1+\delta_i^2}}{2\delta_i}I_{pm} & 0
\end{bmatrix}
& \text{if } \mathrm{Im}(\alpha_i)\neq0,
\end{cases}\label{sz}\\[6pt]
l_{\mathrm{lyap}}^{(i)} &=
\begin{cases}
-\gamma_i I_{pm} & \text{if } \mathrm{Im}(\alpha_i)=0, \\[6pt]
-\sqrt{2}\gamma_i
\begin{bmatrix}
I_{pm} & 0
\end{bmatrix}
& \text{if } \mathrm{Im}(\alpha_i)\neq0,
\end{cases}\label{lz}
\end{align}
where $\gamma_i=\sqrt{-2\mathrm{Re}(\alpha_i)}$ and, for nonreal shifts, $\delta_i=\mathrm{Re}(\alpha_i)/\mathrm{Im}(\alpha_i)$.

The matrices $S_{\mathrm{lyap}}^{(i)}$, $L_{\mathrm{lyap}}^{(i)}$, and $Z_{\mathrm{lyap}}^{(i)}$ are defined recursively by
\begin{align}
S_{\mathrm{lyap}}^{(i)}
=\begin{bmatrix}
S_{\mathrm{lyap}}^{(i-1)} & (L_{\mathrm{lyap}}^{(i-1)})^\top l_{\mathrm{lyap}}^{(i)}\\
0 & s_{\mathrm{lyap}}^{(i)}
\end{bmatrix},\qquad
L_{\mathrm{lyap}}^{(i)}
=\begin{bmatrix}
L_{\mathrm{lyap}}^{(i-1)} & l_{\mathrm{lyap}}^{(i)}
\end{bmatrix},\qquad
Z_{\mathrm{lyap}}^{(i)}
=\begin{bmatrix}
Z_{\mathrm{lyap}}^{(i-1)} & z_i
\end{bmatrix},\label{SLZ}
\end{align}
with $Z_{\mathrm{lyap}}^{(i)}$ satisfying the Sylvester equation
\begin{align}
A^\top Z_{\mathrm{lyap}}^{(i)}
-E^\top Z_{\mathrm{lyap}}^{(i)}S_{\mathrm{lyap}}^{(i)}
+\hat{C}^\top L_{\mathrm{lyap}}^{(i)}=0.
\end{align}
As shown in \cite{zulfiqar2025unified,wolfthesis}, the CF-ADI approximation \cite{benner2013reformulated} to the solution of the Lyapunov equation
\begin{align}
A^\top Q_{\mathrm{lyap}}E
+E^\top Q_{\mathrm{lyap}}A
+\hat{C}^\top\hat{C}=0\label{lyap_eq}
\end{align}
is given by
\[
Q_{\mathrm{lyap}}\approx Z_{\mathrm{lyap}}^{(i)}(Z_{\mathrm{lyap}}^{(i)})^\top.
\]

Define
\begin{align}
\bar{z}_i
=\big(A^\top+\alpha_i E^\top\big)^{-1}
\big(\mathcal{C}_{\perp}^{(i-1)}\big)^\top,
\end{align}
where $\mathcal{C}_{\perp}^{(i)}$ is updated recursively as
\begin{align}
\mathcal{C}_{\perp}^{(i)} &=
\begin{cases}
\mathcal{C}_{\perp}^{(i-1)}+\gamma_i^2\bar{z}_i^\top E,
& \text{if } \mathrm{Im}(\alpha_i)=0,\\
\mathcal{C}_{\perp}^{(i-1)}
+2\gamma_i^2\big(\mathrm{Re}(\bar{z}_i)
+\delta_i\mathrm{Im}(\bar{z}_i)\big)^\top E,
& \text{if } \mathrm{Im}(\alpha_i)\neq0,
\end{cases}\label{C_}
\end{align}
with $\mathcal{C}_{\perp}^{(0)}=\hat{C}$. The CF-ADI block $z_i$ is then computed as \cite{benner2013reformulated}
\begin{align}
z_i &=
\begin{cases}
\gamma_i\bar{z}_i
& \text{if } \mathrm{Im}(\alpha_i)=0,\\
\begin{bmatrix}
\sqrt{2}\gamma_i\big(\mathrm{Re}(\bar{z}_i)
+\delta_i\mathrm{Im}(\bar{z}_i)\big)
&
\sqrt{2}\gamma_i\sqrt{\delta_i^2+1}\mathrm{Im}(\bar{z}_i)
\end{bmatrix}
& \text{if } \mathrm{Im}(\alpha_i)\neq0.
\end{cases}\label{z}
\end{align}
The matrix $Z_{\mathrm{lyap}}^{(i)}$ also satisfies \cite{zulfiqar2025unified,wolfthesis}
\begin{align}
A^\top Z_{\mathrm{lyap}}^{(i)}
+E^\top Z_{\mathrm{lyap}}^{(i)}(S_{\mathrm{lyap}}^{(i)})^\top
+(\mathcal{C}_{\perp}^{(i)})^\top L_{\mathrm{lyap}}^{(i)}=0.\label{V_lyap_sylv2}
\end{align}

Define
\[
\bar{B}_{\mathrm{lyap}}^{(i)}
=(Z_{\mathrm{lyap}}^{(i)})^\top\tilde{B},
\qquad
\hat{A}_{\mathrm{lyap}}^{(i)}
=-S_{\mathrm{lyap}}^{(i)}
+\bar{B}_{\mathrm{lyap}}^{(i)}\hat{Z}L_{\mathrm{lyap}}^{(i)},
\]
and let \(T_{\mathrm{radi}}^{(i)}\) be the unique solution of
\begin{align}
(\hat{A}_{\mathrm{lyap}}^{(i)})^\top T_{\mathrm{radi}}^{(i)}
-T_{\mathrm{radi}}^{(i)}S_r^{(i)}
+(L_{\mathrm{lyap}}^{(i)})^\top\hat{Z}L_r^{(i)}
=0.\label{T_radi_sylv}
\end{align}
Equation \eqref{T_radi_sylv} has the same form as \eqref{radi_sylv}, with $\hat{A}$, $E$, and $\hat{C}$ replaced by $\hat{A}_{\mathrm{lyap}}^{(i)}$, $I$, and $L_{\mathrm{lyap}}^{(i)}$, respectively. Thus, $T_{\mathrm{radi}}^{(i)}$ can be computed recursively in the same way as $W^{(i)}$ in G-RADI. The following theorem shows how $w_i$ can be obtained from $Z_{\mathrm{lyap}}^{(i)}$ using $T_{\mathrm{radi}}^{(i)}$.

\begin{theorem}\label{th2}
Let the assumptions of Theorem \ref{th1} hold, and let $s_{\mathrm{lyap}}^{(i)}$, $l_{\mathrm{lyap}}^{(i)}$, $S_{\mathrm{lyap}}^{(i)}$, $L_{\mathrm{lyap}}^{(i)}$, and $Z_{\mathrm{lyap}}^{(i)}$ be defined by \eqref{sz}--\eqref{V_lyap_sylv2}. Assume that
\[
T_{\mathrm{radi}}^{(i)}
=\begin{bmatrix}
T_{\mathrm{radi}}^{(i-1)} & t_1^{(i)}\\
0 & t_2^{(i)}
\end{bmatrix}
\]
uniquely solves \eqref{T_radi_sylv}. Set
$\hat{B}_{\mathrm{lyap}}^{(i)}=(Z_{\mathrm{lyap}}^{(i)})^\top\hat{B}$
and define
\begin{align}
t_i
=\Bigg(
(\hat{A}_{\mathrm{lyap}}^{(i)})^\top
-\begin{bmatrix}
T_{\mathrm{radi}}^{(i-1)}X_r^{(i-1)}\hat{B}_r^{(i-1)}\\
0
\end{bmatrix}&
\hat{R}^{-1}(\hat{B}_{\mathrm{lyap}}^{(i)})^\top
+\alpha_i I
\Bigg)^{-1}
\Bigg(
(L_{\mathrm{lyap}}^{(i)})^\top
-\begin{bmatrix}
T_{\mathrm{radi}}^{(i-1)}(\hat{C}_r^{(i-1)})^\top\\
0
\end{bmatrix}
\Bigg)\hat{Z}.\label{ti}
\end{align}
Let
\begin{align}
t_{\mathrm{radi}}^{(i)}
=\begin{bmatrix}
t_1^{(i)}\\
t_2^{(i)}
\end{bmatrix}
&=
\begin{cases}
t_i & \text{if } \mathrm{Im}(\alpha_i)=0,\\
\begin{bmatrix}
\mathrm{Re}(t_i) & \mathrm{Im}(t_i)
\end{bmatrix}
& \text{if } \mathrm{Im}(\alpha_i)\neq0.
\end{cases}\label{t}
\end{align}
Then
\[
w_i=Z_{\mathrm{lyap}}^{(i)}t_{\mathrm{radi}}^{(i)},
\qquad
W^{(i)}=Z_{\mathrm{lyap}}^{(i)}T_{\mathrm{radi}}^{(i)}.
\]
\end{theorem}
\begin{proof}
The proof is given in Appendix B.
\end{proof}

Computing $t_{\mathrm{radi}}^{(i)}$ requires a linear solve whose order equals the number of columns of $Z_{\mathrm{lyap}}^{(i)}$ and whose right-hand side has $p+m_1$ columns. This solve is inexpensive as long as the reduced dimension remains much smaller than $n$. G-RADI can therefore be implemented efficiently using CF-ADI as the base algorithm.

The computational cost of Algorithm \ref{alg1} is roughly comparable to that of a single inner Lyapunov solve in Algorithm \ref{alg_nl}, using the $LDL^\top$ variant of the low-rank ADI algorithm \cite{lang2015benefits} to approximate the solution of \eqref{lyap_nl}. By avoiding repeated inner Lyapunov solves, G-RADI can substantially reduce the computational cost relative to the Newton--Kleinman iteration.
\begin{remark}
If G-RADI is used only to compute the gain matrix $\tilde{K}_{\mathrm{gain}}^{(i)}$, the SMW-based implementation requires less memory because $W^{(i)}$ and $X_r^{(i)}$ need not be stored. The iteration requires only $w_i$, $x_r^{(i)}$, $C_{\perp}^{(i)}$, and $\hat{K}^{(i)}$ to proceed. In contrast, using CF-ADI as the base algorithm to compute $w_i$ requires storing $Z_{\mathrm{lyap}}^{(i)}$ and $X_r^{(i)}$.
\end{remark}
\begin{remark}
Since G-RADI only requires $C_{\perp}^{(i)}$ and $\hat{K}^{(i)}$ to start, and in principle does not need the previous history of $W^{(i)}$ and $X_r^{(i)}$, it can be started with $\hat{K}^{(0)}\neq 0$, provided the residual is available in the form $\big(C_{\perp}^{(0)}\big)^\top\hat{Z} C_{\perp}^{(0)}$. Thus, G-RADI can be initialized with a previous approximation when further refinement of the low-rank solution is required.
\end{remark}
\begin{remark}
From Theorem \ref{th2}, one can note that a single Newton--Kleinman iteration with $\hat{K}_0=0$ contains all the information required to stitch a low-rank solution of \eqref{care}, if the Lyapunov equation \eqref{lyap_nl} is approximated via the low-rank ADI method. G-RADI can be seen as a method that extracts such a solution from the low-rank factors of the Lyapunov equation \eqref{lyap_nl} in the Newton--Kleinman iteration with $\hat{K}_0=0$. G-RADI is still an iterative algorithm, but it avoids the outer iterations of the inexact Newton--Kleinman method, in which the Lyapunov equation \eqref{lyap_nl} is solved inexactly and iteratively (inner iterations).
\end{remark}
\subsection{G-RADI as a Projection Algorithm}
G-RADI can be interpreted as a projection-based MOR algorithm under the assumption that $W^{(i)}$ has full column rank and that the pair $(S_r^{(i)},\hat{Z}L_r^{(i)})$ is observable. Let $V^{(i)}$ satisfy $(W^{(i)})^\top E V^{(i)} = I$. The Petrov--Galerkin projection gives the reduced matrices
\begin{align}
\hat{A}_r^{(i)}&=(W^{(i)})^\top \hat{A}V^{(i)}=A_r^{(i)}+\tilde{B}_r^{(i)}\hat{Z}\hat{C}_r^{(i)}=A_r^{(i)}-B_{1,r}^{(i)}R_1^{-1}C_{2,r}^{(i)}, \quad A_r^{(i)}=(W^{(i)})^\top AV^{(i)}\nonumber\\
\tilde{B}_r^{(i)}&=(W^{(i)})^\top \tilde{B}=\begin{bmatrix}0&B_{1,r}^{(i)}\end{bmatrix},\quad
\hat{B}_r^{(i)}=(W^{(i)})^\top \hat{B}=\begin{bmatrix}B_{1,r}^{(i)}&B_{2,r}^{(i)}\end{bmatrix},\quad\hat{C}_r^{(i)}=\hat{C}V^{(i)}=\begin{bmatrix}C_{1,r}^{(i)}\\C_{2,r}^{(i)}\end{bmatrix}.\label{ROM}
\end{align}

Define the transfer functions
\[
\hat{G}(s)=\hat{C}(s E-\hat{A})^{-1}\hat{B}\quad\text{and}\quad \hat{G}_r^{(i)}(s)=\hat{C}_r^{(i)}(s I-\hat{A}_r^{(i)})^{-1}\hat{B}_r^{(i)}.
\]
The connection between Sylvester equations and rational interpolation established in \cite{gallivan2004sylvester} then gives the interpolation conditions
\begin{align}
\hat{G}(-\alpha_i)=\hat{G}_r^{(i)}(-\alpha_i),\qquad i=1,\dots,k.\label{int_cond}
\end{align}

Premultiplying \eqref{radi_sylv} by $(V^{(i)})^\top$ yields
\begin{align}
(\hat{A}_r^{(i)})^\top-S_r^{(i)}+(\hat{C}_r^{(i)})^\top\hat{Z}L_r^{(i)}=0.\nonumber
\end{align}
Consequently, $\hat{A}_r^{(i)} = (S_r^{(i)})^\top - (L_r^{(i)})^\top \hat{Z} \hat{C}_r^{(i)}$ can be parameterized by $\hat{C}_r^{(i)}$ without affecting the interpolation conditions \eqref{int_cond}; cf.\ \cite{wolfthesis,panzerthesis,astolfi2010model}. Varying the free parameter $\hat{C}_r^{(i)}$ is equivalent to varying $V^{(i)}$ while preserving these conditions. Thus, a family of interpolants \eqref{ROM} can be obtained without explicitly specifying $V^{(i)}$.

From a pole-placement perspective, $\hat{C}_r^{(i)}$ acts as a gain matrix that places the poles of $\hat{A}_r^{(i)}$. Since $(S_r^{(i)},\hat{Z}L_r^{(i)})$ is observable, this gain can be chosen to place the poles at
\[
\lambda\Big(-(X_r^{(i)})^{-1}S_r^{(i)}X_r^{(i)}+\hat{B}_r^{(i)}\hat{R}^{-1}(\hat{B}_r^{(i)})^\top X_r^{(i)}\Big).
\]
This ensures that $X_r^{(i)}$ is a stabilizing solution of the projected CARE \eqref{proj_ricc}. The choice $\hat{C}_r^{(i)}=\hat{C}V^{(i)}=L_r^{(i)}X_r^{(i)}$ performs this pole placement. Moreover, since 
\[
C_{\perp}^{(i)}V^{(i)}=\big(\hat{C}-\hat{C}_r^{(i)}(W^{(i)})^\top E\big)V^{(i)}=\hat{C}V^{(i)}-\hat{C}_r^{(i)}=0,
\] the residual $R(\hat{X}^{(i)})$ satisfies the Petrov--Galerkin condition
\[
(V^{(i)})^\top R(\hat{X}^{(i)}) V^{(i)} = 0.
\]
Thus, the G-RADI approximation $\hat{X}^{(i)}= W^{(i)}X_r^{(i)}(W^{(i)})^\top$ is a Petrov--Galerkin approximation of $X$ obtained through rational interpolation-based MOR.
\subsection{G-RADI as a MOR Framework}
Define $\tilde{L}_r^{(i)}=L_r^{(i)}+(\tilde{B}_r^{(i)})^\top$. The Sylvester equation \eqref{radi_sylv} can then be rewritten as
\begin{align}
A^\top W^{(i)}-E^\top W^{(i)} S_r^{(i)}+\hat{C}^\top\hat{Z}\tilde{L}_r^{(i)}&=0.
\end{align}
Define the transfer functions
\[
G(s)=\hat{C}(s E-A)^{-1}\hat{B}\quad\text{and}\quad G_r^{(i)}(s)=\hat{C}_r^{(i)}(s I-A_r^{(i)})^{-1}\hat{B}_r^{(i)}.
\]
Assume that the pair $(S_r^{(i)},\hat{Z}\tilde{L}_r^{(i)})$ is observable. The connection between Sylvester equations and rational interpolation established in \cite{gallivan2004sylvester} then gives the interpolation conditions
\begin{align}
G(-\alpha_i)=G_r^{(i)}(-\alpha_i),\qquad i=1,\dots,k.\label{int_cond_n}
\end{align}

Define the deflated transfer function
\[
G_{\perp}^{(i)}(s)=C_{\perp}^{(i)}(s E-A)^{-1}\hat{B}.
\]
The error $G(s)-G_r^{(i)}(s)$ can be factorized as
\[
G(s)-G_r^{(i)}(s)=\Big(-\hat{C}_r^{(i)}\big(sI-A_r^{(i)}\big)^{-1}(\tilde{L}_r^{(i)})^\top\hat{Z}+I\Big)G_{\perp}^{(i)}(s);
\]
cf.\ \cite{wolfthesis}. If $G_r^{(i)}(s)$ captures all observable poles of $G(s)$, these poles are deflated in $G_{\perp}^{(i)}(s)$. In this case, $G_{\perp}^{(i)}(s)$ contains no observable poles, and $G_r^{(i)}(s)$ recovers $G(s)$ exactly \cite{wolfthesis}.

Thus, G-RADI can be interpreted as a MOR algorithm that recursively interpolates at the mirror images of the ADI shifts. It recovers the original transfer function $G(s)$ when no observable dynamics remain in the deflated transfer function $G_{\perp}^{(i)}(s)$.

Depending on the structure of the CARE \eqref{gen_ricc}, the ROM $G_r^{(i)}(s)$ can preserve system-theoretic properties of $G(s)$. For example, suppose that $X>0$ solves the CARE associated with a state-space realization of a positive-real or bounded-real transfer function \cite{phillips2002guaranteed}. If $X_r^{(i)}>0$, the ROM $G_r^{(i)}(s)$ preserves positive-realness or bounded-realness, respectively, while interpolating $G(s)$ at $-\alpha_i$. Thus, G-RADI provides a MOR framework for constructing interpolatory ROMs with these properties. For further details, see Section~3.8 of~\cite{zulfiqar2025unified}.
\subsection{Automatic Shift Generation}
Shift generation methods that seek optimal shifts with respect to a norm of the residual $R(\hat{X}^{(i)})$ are generally expensive and therefore of limited practical interest \cite{benner2014self}. Viewing G-RADI as a rational interpolation algorithm enables the use of guidelines for selecting good interpolation points as ADI shifts \cite{beattie2017chapter}. In rational interpolation, choosing the mirror images of the most controllable and most observable poles of $G(s)$ (or $G_{\perp}^{(i)}(s)$) gives good accuracy, since the error $\|G(s)-G_r^{(i)}(s)\|_{\mathcal{H}_2}$ is dominated by the mismatch at these complex frequencies \cite{gugercin2008h_2}. Using dominant poles as shifts in low-rank ADI methods also leads to a rapid decrease in the residuals \cite{saak2009efficient,benner2014self}. Although the dominant poles are generally unavailable, Ritz values of $AE^{-1}$ can be computed using the interpolation basis accumulated by G-RADI.

Define
\[
W_{\mathrm{proj}} = \mathrm{orth}\Big( \begin{bmatrix} w_1, \dots, w_i \end{bmatrix} \Big),
\]
with implicit restart: when the number of columns exceeds a user-defined limit, all previously accumulated $w_i$ blocks are discarded and the accumulation restarts. Implicit restart is widely used in eigenvalue problems \cite{saad2011numerical,rommes2008modal} and keeps the eigenvalue decomposition required to compute the Ritz values computationally feasible.

Project $(E, A, C_{\perp}^{(i)})$ using $W_{\mathrm{proj}}$:
\[
E_{\mathrm{proj}} = (W_{\mathrm{proj}})^\top E W_{\mathrm{proj}}, \quad
A_{\mathrm{proj}} = (W_{\mathrm{proj}})^\top A W_{\mathrm{proj}}, \quad
C_{\mathrm{proj}} = C_{\perp}^{(i)} W_{\mathrm{proj}}.
\]

Compute the eigenvalue decomposition
\[
A_{\mathrm{proj}} E_{\mathrm{proj}}^{-1} = T \, \mathrm{diag}\big( \lambda_1, \dots, \lambda_r \big) T^{-1}.
\]
Define $r_j = C_{\mathrm{proj}}E_{\mathrm{proj}}^{-1} T(:, j)$. The most observable pole of $A_{\mathrm{proj}} E_{\mathrm{proj}}^{-1}$ is the pole $\lambda_j$ with the largest normalized residue
\[
\phi_j = \frac{\| r_j \|_2^2}{| \mathrm{Re}(\lambda_j) |},
\]
cf.\ \cite{rommes2007methods,mengi2022large}. Sort the columns of $T$ and the corresponding eigenvalues $\lambda_j$ in descending order of $\phi_j$. The Ritz value with the largest $\phi_j$ can then be used as the next ADI shift. If its real part is nonnegative, the previous shift can be repeated.
\section{Numerical Results}\label{sec:numerics}
This section presents numerical examples that compare G-RADI (Algorithm \ref{alg1}), implemented using both the SMW formula and CF-ADI, against the Newton--Kleinman iteration (Algorithm~\ref{alg_nl}). Three models are considered: a modest-order model used to validate the implementation against MATLAB's \texttt{icare}, and two large-scale sparse models. The MATLAB code and data needed to reproduce the results are available at \cite{mycode}.

\subsection{Simulation Setup}
All experiments were run in MATLAB R2025b on a Windows 11 laptop with an Intel Core Ultra 9 285H processor (2.9 GHz) and 32 GB of RAM.

For an approximation $\hat{X}$ of $X$, denote the residual by
\[
R(\hat{X})=\hat{A}^\top \hat{X}E + E^\top \hat{X} \hat{A} - E^\top \hat{X} \hat{B} \hat{R}^{-1} \hat{B}^\top \hat{X} E + \hat{C}^\top \hat{Z} \hat{C}.
\]
The approximation quality was measured by the normalized residual norm
\[
\dfrac{\|R(\hat{X})\|_2}{\|\hat{C}^\top \hat{Z} \hat{C}\|_2}.
\]

The first ADI shift was set to $-0.001$ for both G-RADI and the $LDL^\top$ low-rank ADI method of~\cite{lang2015benefits}, which was used within the Newton--Kleinman iteration. The remaining shifts were generated with the strategy described in the previous section. G-RADI was allowed at most $100$ iterations with tolerance $\tau=10^{-12}$, and it converged within a few minutes in all cases.

The Newton--Kleinman iteration was limited to $100$ outer iterations, and the inner $LDL^\top$ low-rank ADI solve of~\cite{lang2015benefits} was also limited to $100$ iterations. Both the inner and outer tolerances were set to $\tau=10^{-12}$. The Newton--Kleinman iteration also converged in all tests. The Newton--Kleinman iteration requires a stabilizing initial gain $\hat{K}_0$. We chose examples in which $AE^{-1}$ is Hurwitz, so that the zero gain $\hat{K}_0=0$ could be used. G-RADI does not have this restriction.
\subsection{Comparison with MATLAB's \texttt{icare}}
The first model was generated using the procedure of~\cite{zulfiqar2026numerical}, with $E,A\in\mathbb{R}^{200\times 200}$, $B\in\mathbb{R}^{200\times 2}$, and $C\in\mathbb{R}^{2\times 200}$. The two columns of $B$ were taken as $B_1$ and $B_2$, and the two rows of $C$ as $C_1$ and $C_2$, with $R_1=-0.6572$, $R_2=0.9744$, and $Z=-0.2324$. By construction, only six poles of $E^{-1}A$ are dominant, and the remaining modes are nearly unobservable. Thus, $X$ has six dominant singular values, making a low-rank approximation appropriate.

A reference stabilizing solution $X_{\text{icare}}$ and the corresponding gain $K_{\text{icare}}$ were computed using MATLAB's \texttt{icare}. The residual satisfied $\|R(X_{\text{icare}})\|_2=2.7236\times 10^{-14}$. G-RADI was run using both the SMW formula and CF-ADI, with the implicit-restart limit set to $20$ columns. Both G-RADI and the Newton--Kleinman iteration converged. Figure~\ref{fig1} shows the residual histories.

\begin{figure}[!h]
  \centering
  \includegraphics[width=12cm]{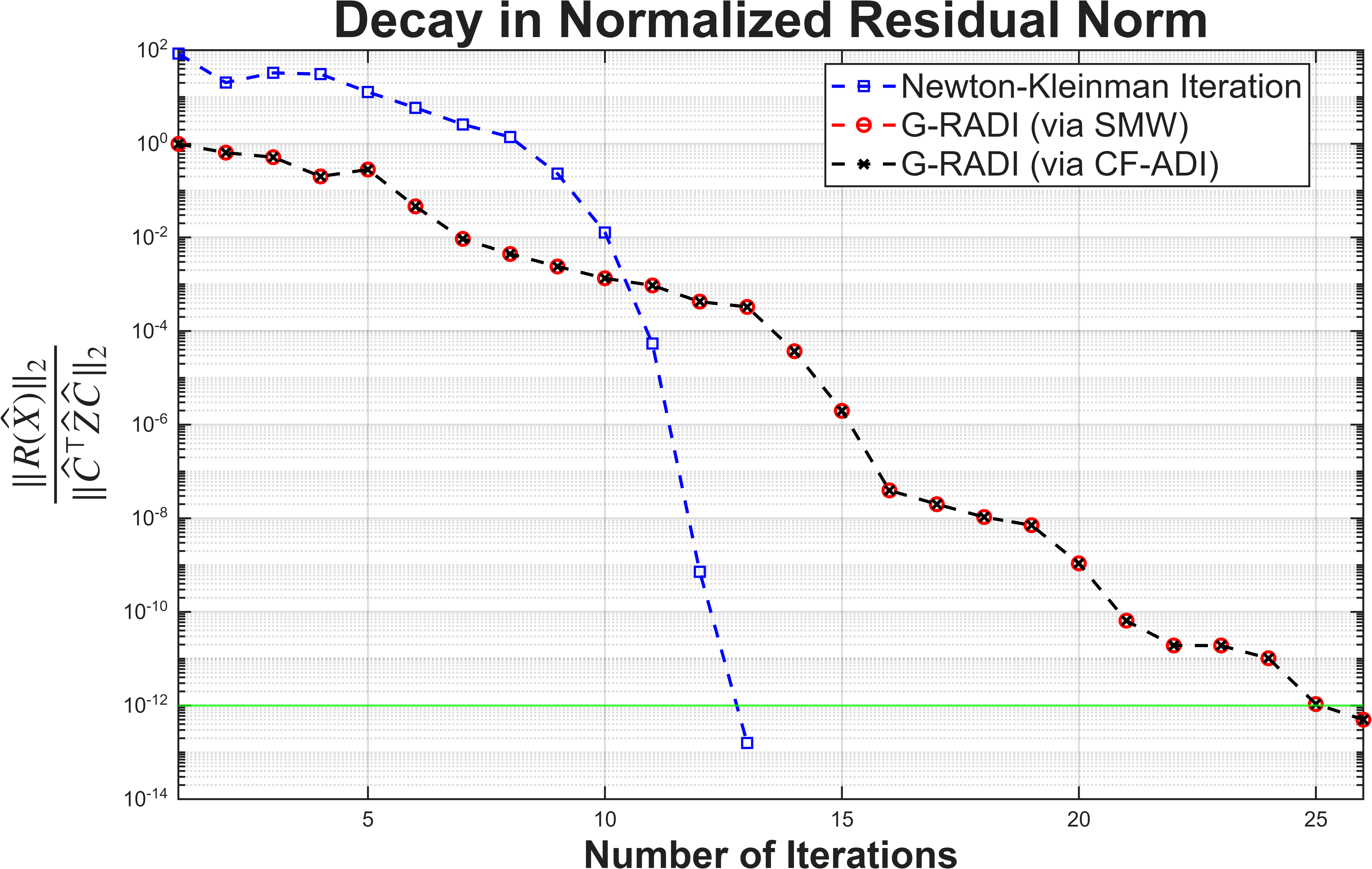}
  \caption{Normalized residual norm $\frac{\|R(\hat{X})\|_2}{\|\hat{C}^\top \hat{Z} \hat{C}\|_2}$ for the model of order $200$}\label{fig1}
\end{figure}

Table~\ref{tab1} reports the differences in the solution and gain relative to MATLAB's \texttt{icare}. G-RADI and MATLAB's \texttt{icare} computed the same stabilizing solution and gain, whereas the Newton--Kleinman iteration converged to a different solution. Figure~\ref{fig2} shows the closed-loop poles $\lambda(A_{\mathrm{cl}})$ corresponding to each solution and gain. For the Newton--Kleinman solution and gain, three closed-loop poles lay in the right half-plane, indicating that the Newton--Kleinman iteration did not produce a stabilizing solution. By contrast, all closed-loop poles obtained with G-RADI and MATLAB's \texttt{icare} lay in the left half-plane, confirming that both yielded stabilizing solutions of the CARE~\eqref{gen_ricc}. This example illustrates the practical value of the pole-placement property used by G-RADI, which guarantees that the projected CARE~\eqref{proj_ricc} always admits a stabilizing solution.

\begin{table}[!h]
\centering
\caption{Relative differences in the solution and gain with respect to MATLAB's \texttt{icare}}\label{tab1}
\begin{tabular}{|c|c|c|}
\hline
Method & $\dfrac{\|X_{\text{icare}}-\hat{X}\|_2}{\|X_{\text{icare}}\|_2}$ & $\dfrac{\|K_{\text{icare}}-\hat{K}\|_2}{\|K_{\text{icare}}\|_2}$ \\ \hline
 & & \\
G-RADI (via SMW) & $3.9662\times 10^{-11}$ & $2.4018\times 10^{-11}$ \\
 & & \\
G-RADI (via CF-ADI) & $4.2585\times 10^{-11}$ & $2.0977\times 10^{-11}$ \\
 & & \\
Newton--Kleinman iteration & $16.5272$ & $3.9463$ \\ \hline
\end{tabular}
\end{table}

\begin{figure}[!h]
  \centering
  \includegraphics[width=12cm]{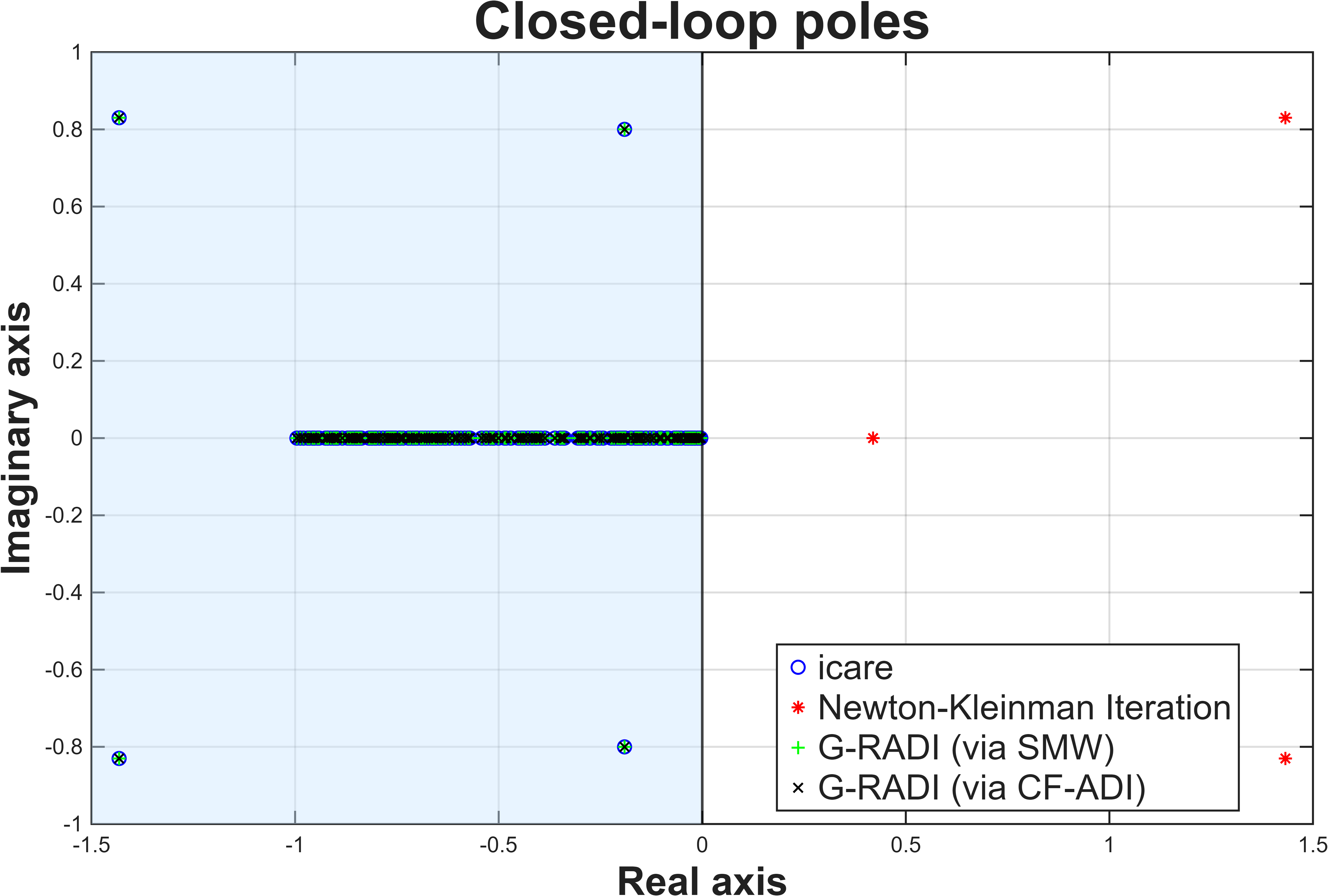}
  \caption{Closed-loop poles $\lambda(A_{\text{cl}})$}\label{fig2}
\end{figure}
\subsection{Steel Profile Model}
The second example was the rail model, a semi-discretized heat transfer model for the optimal cooling of steel profiles~\cite{benner2005semi}. The model was taken from~\cite{saak2021mm} and has order $1\,265\,537$. Here, $E,A\in\mathbb{R}^{1265537\times 1265537}$, $B\in\mathbb{R}^{1265537\times 7}$, and $C\in\mathbb{R}^{6\times 1265537}$. The first four columns of $B$ formed $B_1$ and the remaining three $B_2$; the first two rows of $C$ formed $C_1$ and the remaining four $C_2$. The weighting matrices were
\[
Z =\begin{bmatrix}0.1631 &   0.8128\\
    0.8128  &  0.2355\end{bmatrix},\quad
R_1 =\begin{bmatrix}0.9571  &  0.5263  &  0.6276  &  0.3459\\
    0.5263  &  0.5816 &   0.5266  &  0.7908\\
    0.6276  &  0.5266 &   0.2404  &  0.4062\\
    0.3459  &  0.7908 &   0.4062  &  0.7139\end{bmatrix},\quad
R_2 =\begin{bmatrix}0.7223  &  0.7430  &  0.8722\\
    0.7430  &  0.1107 &   0.9064\\
    0.8722   & 0.9064  &  0.1739\end{bmatrix}.
\]
Since $\lambda(Z)=(-0.6143,1.0129)$, $\lambda(R_1)=(-0.2230,-0.1369,0.5898,2.2631)$, and $\lambda(R_2)=(-0.7712$,$-0.2636$, $2.0417)$, all three matrices are indefinite. RADI~\cite{benner2018radi} is not applicable to this problem, and MATLAB's \texttt{icare} is infeasible at this order. G-RADI was therefore compared with the Newton--Kleinman iteration.

For both G-RADI implementations, and for the inner $LDL^\top$ low-rank ADI solve of~\cite{lang2015benefits} used within the Newton--Kleinman iteration, the implicit-restart limit was set to $18$ columns. Both G-RADI and the Newton--Kleinman iteration converged, and their residual histories are shown in Figures~\ref{fig3} and~\ref{fig4}, respectively. Table~\ref{tab2} reports the simulation times. G-RADI required substantially less computational time than the Newton--Kleinman iteration. Avoiding the SMW formula in G-RADI yielded no speed-up in this example, and the CF-ADI-based implementation was slightly slower because it retained additional factors. A speed-up is expected only when $p$, $m_1$, and $m_2$ are large enough for the cost of the additional right-hand sides in the SMW solve to dominate.
\begin{figure}[!h]
  \centering
  \includegraphics[width=12cm]{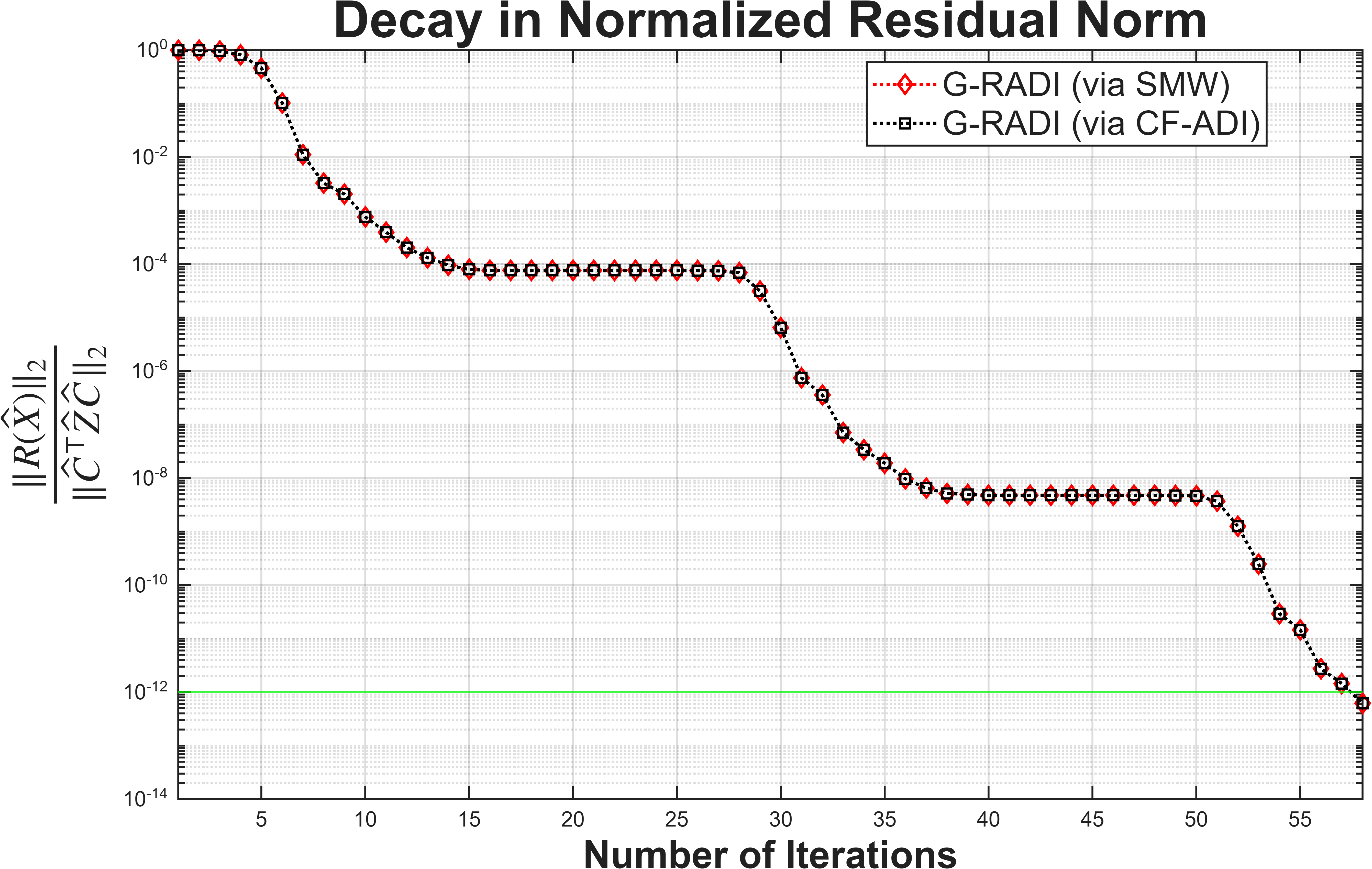}
  \caption{G-RADI: Normalized residual norm $\frac{\|R(\hat{X})\|_2}{\|\hat{C}^\top \hat{Z} \hat{C}\|_2}$}\label{fig3}
\end{figure}

\begin{figure}[!h]
  \centering
  \includegraphics[width=12cm]{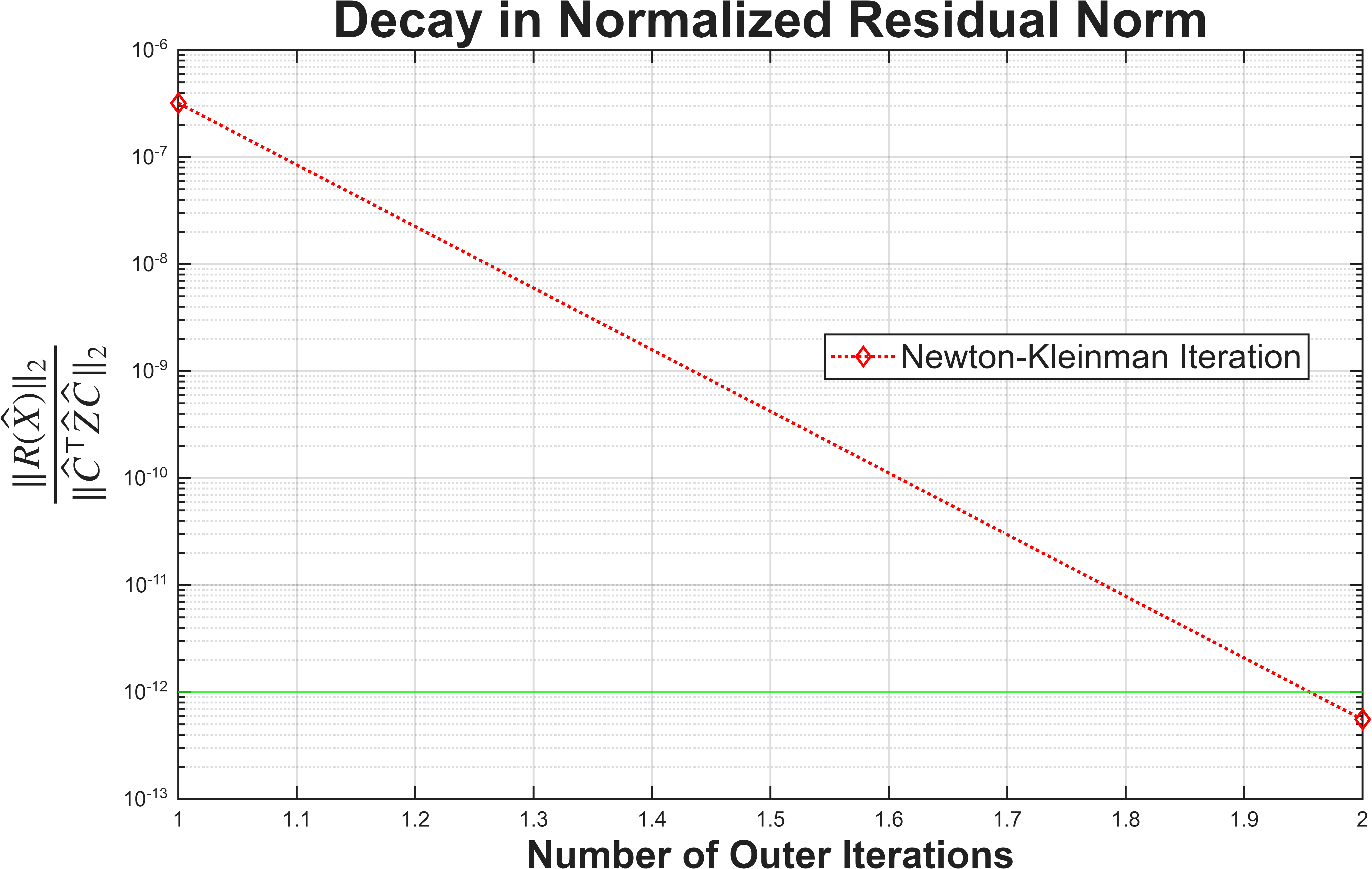}
  \caption{Newton--Kleinman iteration: Normalized residual norm $\frac{\|R(\hat{X})\|_2}{\|\hat{C}^\top \hat{Z} \hat{C}\|_2}$}\label{fig4}
\end{figure}

\begin{table}[!h]
\centering
\caption{Simulation times for the steel profile model}\label{tab2}
\begin{tabular}{|c|c|}
\hline
Method & Simulation time (s) \\ \hline
G-RADI (via SMW) & $224.2431$ \\ \hline
G-RADI (via CF-ADI) & $252.7239$ \\ \hline
Newton--Kleinman iteration & $614.8708$ \\ \hline
\end{tabular}
\end{table}
\subsection{RLC Ladder}
The third example is the RLC ladder network from~\cite{gugercin2004survey} with $5\times 10^6$ segments. Here, $E,A\in\mathbb{R}^{10^7\times 10^7}$, $B\in\mathbb{R}^{10^7\times 1}$, $C\in\mathbb{R}^{1\times 10^7}$, and $D\in\mathbb{R}^{1\times 1}$. Five CAREs associated with this model are solved using both the SMW-based G-RADI implementation and the Newton--Kleinman iteration. The implicit-restart limit is set to \(8\) columns for G-RADI and for the inner \(LDL^\top\) low-rank ADI solve of~\cite{lang2015benefits} employed within the Newton--Kleinman iteration. Each CARE is listed below together with the matrix choices that cast it in the form~\eqref{gen_ricc}.
\begin{enumerate}
\item The CARE
\[
A^\top Q_{\mathrm{ricc}}^{\prime}E + E^\top Q_{\mathrm{ricc}}^{\prime}A + E^\top Q_{\mathrm{ricc}}^{\prime} B B^\top Q_{\mathrm{ricc}}^{\prime} E + C^\top C = 0
\]
has an indefinite quadratic term~\cite{benner2023low}. It takes the form~\eqref{gen_ricc} with $B_1 = B$, $B_2 = 0$, $R_1 = -I$, $R_2 = I$, $C_1 = C$, $Z = I$, and $C_2 = 0$.

\item The CARE
\[
A^\top Q_{\mathrm{pr}}E + E^\top Q_{\mathrm{pr}}A + (C - B^\top Q_{\mathrm{pr}}E)^\top (D + D^\top)^{-1} (C - B^\top Q_{\mathrm{pr}}E) = 0
\]
arises in positive-real balanced truncation~\cite{green1988balanced}. It takes the form~\eqref{gen_ricc} with $B_1 = -B$, $B_2 = 0$, $R_1 = -(D + D^\top)$, $R_2 = I$, $C_1 = 0$, $Z = I$, and $C_2 = C$.

\item The CARE
\[
A^\top Q_{\mathrm{br}}E + E^\top Q_{\mathrm{br}}A + C^\top C + (B^\top Q_{\mathrm{br}}E + D^\top C)^\top (I - D^\top D)^{-1} (B^\top Q_{\mathrm{br}}E + D^\top C) = 0
\]
arises in bounded-real balanced truncation~\cite{opdenacker1988contraction}. It takes the form~\eqref{gen_ricc} with $B_1 = B$, $B_2 = 0$, $R_1 = -(I - D^\top D)$, $R_2 = I$, $C_1 = C$, $Z = I$, and $C_2 = D^\top C$.

\item The CARE
\[
A^\top Q_{\mathrm{LQG}}E + E^\top Q_{\mathrm{LQG}}A - (B^\top Q_{\mathrm{LQG}}E + D^\top C)^\top (R + D^\top D)^{-1} (B^\top Q_{\mathrm{LQG}}E + D^\top C) + C^\top Q C = 0
\]
arises in LQG controller design and LQG balanced truncation~\cite{jonckheere1983new}, with weights $Q=0.2769$ and $R=0.6557$. It takes the form~\eqref{gen_ricc} with $B_1 = B$, $B_2 = 0$, $R_1 = R + D^\top D$, $R_2 = I$, $C_1 = C$, $Z = Q$, and $C_2 = D^\top C$.

\item The CARE
\[
A^\top Q_{\infty}E + E^\top Q_{\infty}A - E^\top Q_{\infty}\left( B R^{-1} B^\top - \frac{1}{\gamma^2} B B^\top \right) Q_{\infty} E + C^\top Q C = 0
\]
arises in $\mathcal{H}_\infty$ controller design and $\mathcal{H}_\infty$ balanced truncation~\cite{mustafa1990controller}. Here, $\gamma=1.5$, and the weights $Q$ and $R$ are as in the previous item. It takes the form~\eqref{gen_ricc} with $B_1 = B$, $B_2 = \frac{1}{\gamma}B$, $R_1 = R$, $R_2 = I$, $C_1 = C$, $Z = Q$, and $C_2 = 0$.
\end{enumerate}

Figures~\ref{fig5} and~\ref{fig6} show the residual histories of G-RADI (via SMW) and the Newton--Kleinman iteration, respectively, for the five CAREs. Both methods converged for all five CAREs. Table~\ref{tab3} reports the computational times. G-RADI required substantially less computational time than the Newton--Kleinman iteration in all five cases.

\begin{figure}[!h]
  \centering
  \includegraphics[width=12cm]{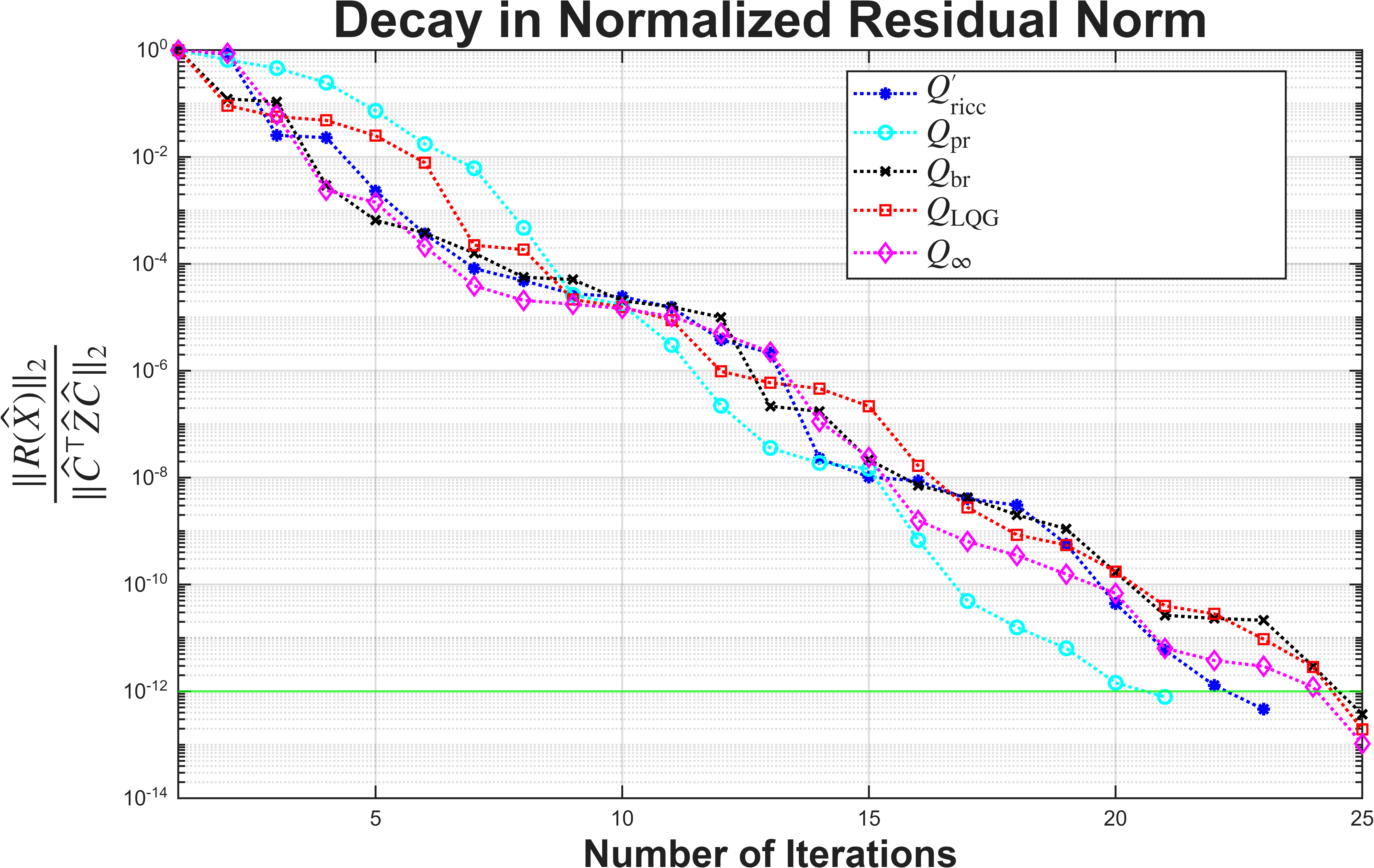}
  \caption{G-RADI: Normalized residual norm $\frac{\|R(\hat{X})\|_2}{\|\hat{C}^\top \hat{Z} \hat{C}\|_2}$}\label{fig5}
\end{figure}

\begin{figure}[!h]
  \centering
  \includegraphics[width=12cm]{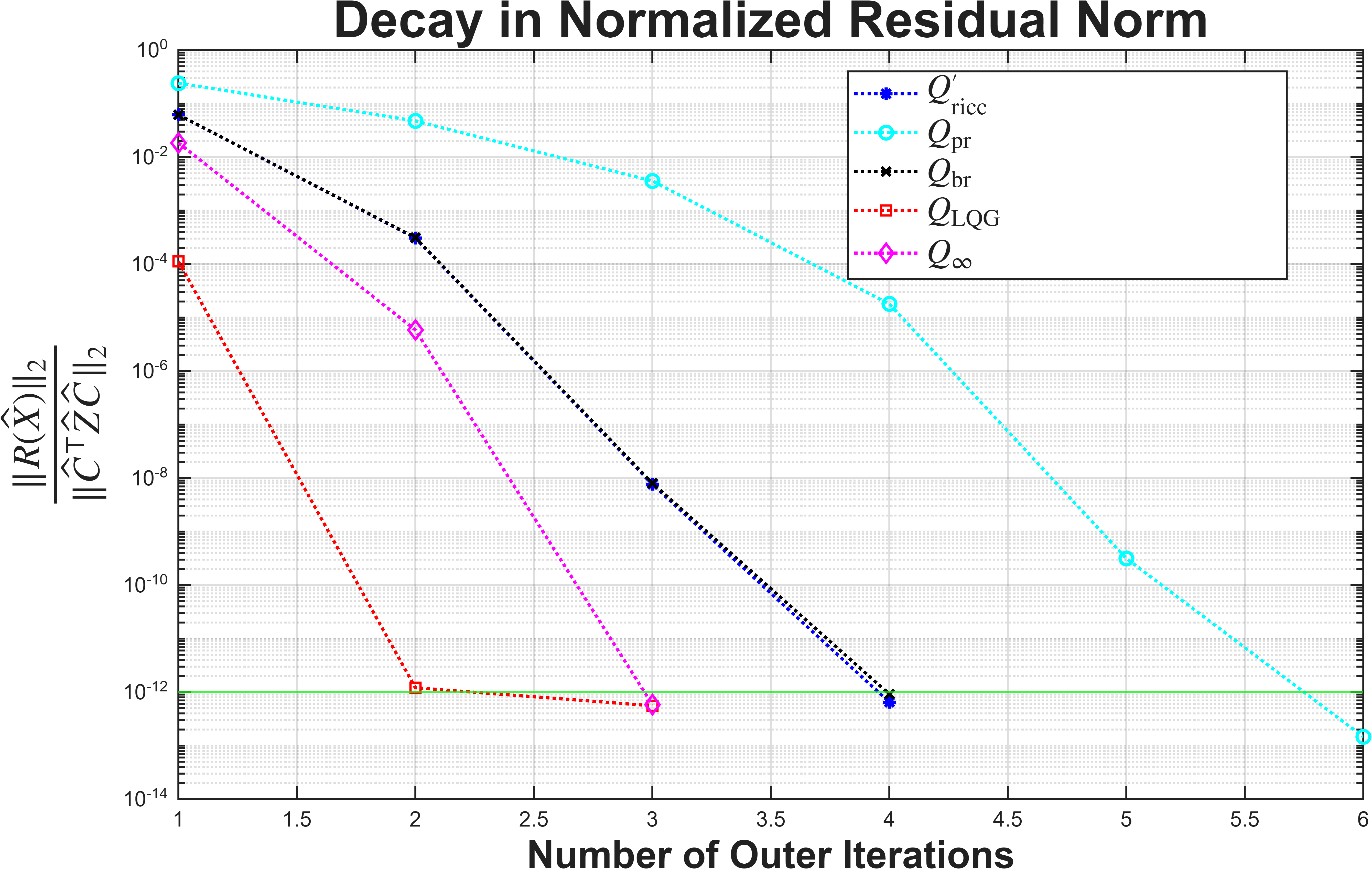}
  \caption{Newton--Kleinman iteration: Normalized residual norm $\frac{\|R(\hat{X})\|_2}{\|\hat{C}^\top \hat{Z} \hat{C}\|_2}$}\label{fig6}
\end{figure}

\begin{table}[!h]
\centering
\caption{Computational times for the RLC ladder model}\label{tab3}
\begin{tabular}{|c|cc|}
\hline
\multirow{2}{*}{CARE} & \multicolumn{2}{c|}{Computational time (s)} \\ \cline{2-3}
 & \multicolumn{1}{c|}{G-RADI} & Newton--Kleinman iteration \\ \hline
$Q_{\mathrm{ricc}}^{\prime}$ & \multicolumn{1}{c|}{$73.7448$} & $388.8652$ \\
$Q_{\mathrm{pr}}$ & \multicolumn{1}{c|}{$75.4037$} & $700.9771$ \\
$Q_{\mathrm{br}}$ & \multicolumn{1}{c|}{$94.8469$} & $509.1731$ \\
$Q_{\mathrm{LQG}}$ & \multicolumn{1}{c|}{$80.9202$} & $289.6576$ \\
$Q_{\infty}$ & \multicolumn{1}{c|}{$64.0383$} & $379.4963$ \\ \hline
\end{tabular}
\end{table}
\section{Conclusion}\label{sec:conclusion}
A generalized low-rank RADI algorithm (G-RADI) has been developed for CAREs in general form with possibly indefinite weighting matrices. G-RADI recursively constructs low-rank solution factors while maintaining a fixed low-rank representation of the residual. Its SMW- and CF-ADI-based implementations have been shown to be equivalent. The projection interpretation links G-RADI to rational-interpolation-based MOR and provides guidelines for automatic shift generation. In this strategy, Ritz values computed from the accumulated interpolation basis approximate the dominant poles. Numerical examples validated the implementation against MATLAB's \texttt{icare} and demonstrated its applicability to large-scale sparse problems with up to $10^7$ states. G-RADI converged in all examples and required substantially less computational time than the Newton--Kleinman iteration. These results show that G-RADI is an efficient solver for large-scale CAREs in general form.
\section*{Appendix A: Proof of Theorem~\ref{th1}}
\begin{proof}
\emph{Part 1.} Post-multiplying \eqref{radi_lyap} by \(X_r^{(i)}\) gives
\[
-(S_r^{(i)})^\top-(X_r^{(i)})^{-1}S_r^{(i)}X_r^{(i)}
+(L_r^{(i)})^\top\hat{Z}\hat{C}_r^{(i)}
+\hat{B}_r^{(i)}\hat{R}^{-1}(\hat{B}_r^{(i)})^\top X_r^{(i)}=0.
\]
By the definition of \(\hat{A}_r^{(i)}\), this yields
\begin{align}
\hat{A}_r^{(i)}
&=-(X_r^{(i)})^{-1}S_r^{(i)}X_r^{(i)}
+\hat{B}_r^{(i)}\hat{R}^{-1}(\hat{B}_r^{(i)})^\top X_r^{(i)},\nonumber\\
\hat{A}_{\mathrm{cl}}^{(i)}
&=-(X_r^{(i)})^{-1}S_r^{(i)}X_r^{(i)}.\nonumber
\end{align}
Thus, \(\hat{A}_{\mathrm{cl}}^{(i)}\) is similar to \(-S_r^{(i)}\) and has the eigenvalues \(\alpha_1,\ldots,\alpha_i\), each with multiplicity \(p+m_1\).

The reduced CARE~\eqref{proj_ricc} follows from
\begin{align}
&(\hat{A}_r^{(i)})^\top X_r^{(i)}
+X_r^{(i)}\hat{A}_r^{(i)}
-X_r^{(i)}\hat{B}_r^{(i)}\hat{R}^{-1}(\hat{B}_r^{(i)})^\top X_r^{(i)}
+(\hat{C}_r^{(i)})^\top\hat{Z}\hat{C}_r^{(i)}\nonumber\\
&\quad=(\hat{A}_r^{(i)})^\top X_r^{(i)}
-S_r^{(i)}X_r^{(i)}
+(\hat{C}_r^{(i)})^\top\hat{Z}\hat{C}_r^{(i)}\nonumber\\
&\quad=X_r^{(i)}\Big(
-(S_r^{(i)})^\top(X_r^{(i)})^{-1}
-(X_r^{(i)})^{-1}S_r^{(i)}
+(L_r^{(i)})^\top\hat{Z}L_r^{(i)}
+\hat{B}_r^{(i)}\hat{R}^{-1}(\hat{B}_r^{(i)})^\top
\Big)X_r^{(i)}\nonumber\\
&\quad=0,\nonumber
\end{align}
where the last equality follows from \eqref{radi_lyap}. Since all ADI shifts lie in \(\mathbb{C}_{-}\), the matrix \(\hat{A}_{\mathrm{cl}}^{(i)}\) is Hurwitz. Hence, \(X_r^{(i)}\) is a stabilizing solution of \eqref{proj_ricc}.

\emph{Part 2.} Expanding
\[
C_{\perp}^{(i)}
=\hat{C}-L_r^{(i)}X_r^{(i)}(W^{(i)})^\top E
\]
gives
\begin{align}
\big(C_{\perp}^{(i)}\big)^\top\hat{Z}C_{\perp}^{(i)}
&=\hat{C}^\top\hat{Z}\hat{C}
-\hat{C}^\top\hat{Z}L_r^{(i)}X_r^{(i)}(W^{(i)})^\top E\nonumber\\
&\quad-E^\top W^{(i)}X_r^{(i)}(L_r^{(i)})^\top\hat{Z}\hat{C}\nonumber\\
&\quad+E^\top W^{(i)}X_r^{(i)}(L_r^{(i)})^\top
\hat{Z}L_r^{(i)}X_r^{(i)}(W^{(i)})^\top E.\nonumber
\end{align}
Substituting this identity into \(R(\hat{X}^{(i)})\) and using
\(\hat{B}_r^{(i)}=(W^{(i)})^\top\hat{B}\) yields
\begin{align}
R(\hat{X}^{(i)})
&=\Big(\hat{A}^\top W^{(i)}
+\hat{C}^\top\hat{Z}L_r^{(i)}\Big)
X_r^{(i)}(W^{(i)})^\top E\nonumber\\
&\quad+E^\top W^{(i)}X_r^{(i)}
\Big((W^{(i)})^\top\hat{A}
+(L_r^{(i)})^\top\hat{Z}\hat{C}\Big)\nonumber\\
&\quad-E^\top W^{(i)}X_r^{(i)}
\Big((L_r^{(i)})^\top\hat{Z}L_r^{(i)}
+\hat{B}_r^{(i)}\hat{R}^{-1}(\hat{B}_r^{(i)})^\top\Big)
X_r^{(i)}(W^{(i)})^\top E\nonumber\\
&\quad+\big(C_{\perp}^{(i)}\big)^\top\hat{Z}C_{\perp}^{(i)}.\nonumber
\end{align}
Applying \eqref{radi_sylv} and \eqref{radi_lyap} gives
\begin{align}
R(\hat{X}^{(i)})
&=E^\top W^{(i)}
\Big(
S_r^{(i)}X_r^{(i)}
+X_r^{(i)}(S_r^{(i)})^\top
-X_r^{(i)}
\big((S_r^{(i)})^\top(X_r^{(i)})^{-1}
+(X_r^{(i)})^{-1}S_r^{(i)}\big)
X_r^{(i)}
\Big)(W^{(i)})^\top E\nonumber\\
&\quad+\big(C_{\perp}^{(i)}\big)^\top\hat{Z}C_{\perp}^{(i)}\nonumber\\
&=\big(C_{\perp}^{(i)}\big)^\top\hat{Z}C_{\perp}^{(i)}.\nonumber
\end{align}
Since \(C_{\perp}^{(i)}\) has \(p+m_1\) rows, the residual has rank at most \(p+m_1\).

\emph{Part 3.} Substituting the approximation
\(X\approx W^{(i)}X_r^{(i)}(W^{(i)})^\top\)
into the definition of \(K_{\mathrm{gain}}\) gives
\[
K_{\mathrm{gain}}\approx\tilde{K}_{\mathrm{gain}}^{(i)}
=R_1^{-1}\Big(
B_1^\top W^{(i)}X_r^{(i)}(W^{(i)})^\top E+C_2
\Big).
\]
The block-diagonal structure of \(X_r^{(i)}\) implies
\begin{align}
\tilde{K}_{\mathrm{gain}}^{(i)}
&=R_1^{-1}C_2
+R_1^{-1}B_1^\top W^{(i-1)}X_r^{(i-1)}(W^{(i-1)})^\top E
+R_1^{-1}B_1^\top w_ix_r^{(i)}w_i^\top E\nonumber\\
&=\tilde{K}_{\mathrm{gain}}^{(i-1)}
+R_1^{-1}B_1^\top w_ix_r^{(i)}w_i^\top E,\nonumber
\end{align}
with \(\tilde{K}_{\mathrm{gain}}^{(0)}=R_1^{-1}C_2\).
\end{proof}
\section*{Appendix B: Proof of Theorem~\ref{th2}}
\begin{proof}
The proof is given for real ADI shifts; the complex case follows analogously.

To establish $w_i = Z_{\mathrm{lyap}}^{(i)}t_{\mathrm{radi}}^{(i)}$, it suffices to verify the following chain of equalities:
\begingroup\interdisplaylinepenalty=10000
\begin{align}
&\Big(\hat{A}^\top-E^\top W^{(i-1)}X_r^{(i-1)}(W^{(i-1)})^\top\hat{B}\hat{R}^{-1}\hat{B}^\top+\alpha_iE^\top\Big)w_i\nonumber\\
&\quad=(\hat{A}^\top-E^\top W^{(i-1)}X_r^{(i-1)}(W^{(i-1)})^\top\hat{B}\hat{R}^{-1}\hat{B}^\top+\alpha_iE^\top)Z_{\mathrm{lyap}}^{(i)}t_{\mathrm{radi}}^{(i)}\nonumber\\
&\quad=\big(\hat{C}-L_r^{(i-1)}X_r^{(i-1)}(W^{(i-1)})^\top E\big)^\top\hat{Z}\nonumber\\
&\quad=\big(\hat{C}^\top \hat{Z}-E^\top W^{(i-1)}X_r^{(i-1)}(L_r^{(i-1)})^\top \hat{Z}\big)\nonumber\\
&\quad=(C_{\perp}^{(i-1)})^\top\hat{Z}.\nonumber
\end{align}
\endgroup
First, note that
\[
\begin{bmatrix}a_{11}&a_{12}\\a_{21}&a_{22}\end{bmatrix}\begin{bmatrix}t_1^{(i)}\\t_2^{(i)}\end{bmatrix}=\begin{bmatrix}(L_{\mathrm{lyap}}^{(i-1)})^\top\hat{Z}-T_{\mathrm{radi}}^{(i-1)}X_r^{(i-1)}(L_r^{(i-1)})^\top\hat{Z}\\-\gamma_i\hat{Z}\end{bmatrix},
\]
where
\begingroup\interdisplaylinepenalty=10000
\begin{align}
a_{11}&=-(S_{\mathrm{lyap}}^{(i-1)})^\top+(L_{\mathrm{lyap}}^{(i-1)})^\top\hat{Z}(\bar{B}_{\mathrm{lyap}}^{(i-1)})^\top-T_{\mathrm{radi}}^{(i-1)}X_r^{(i-1)}\hat{B}_r^{(i-1)}\hat{R}^{-1}(\hat{B}_{\mathrm{lyap}}^{(i-1)})^\top+\alpha_iI,\nonumber\\
a_{12}&=(L_{\mathrm{lyap}}^{(i-1)})^\top\hat{Z}\tilde{B}^\top z_i-T_{\mathrm{radi}}^{(i-1)}X_r^{(i-1)}\hat{B}_r^{(i-1)}\hat{R}^{-1}\hat{B}^\top z_i,\nonumber\\
a_{21}&=\gamma_i\big(L_{\mathrm{lyap}}^{(i-1)}-\hat{Z}(\bar{B}_{\mathrm{lyap}}^{(i-1)})^\top\big),\nonumber\\
a_{22}&=-\gamma_i^2 I-\gamma_i\hat{Z}\tilde{B}^\top z_i.\nonumber
\end{align}
\endgroup
Thus,
\begingroup\interdisplaylinepenalty=10000
\begin{align}
a_{11}t_1^{(i)}+a_{12}t_2^{(i)}&=(L_{\mathrm{lyap}}^{(i-1)})^\top\hat{Z}-T_{\mathrm{radi}}^{(i-1)}X_r^{(i-1)}(L_r^{(i-1)})^\top\hat{Z},\label{obsv1}\\
a_{21}t_1^{(i)}+a_{22}t_2^{(i)}&=-\gamma_i\hat{Z}.
\end{align}
\endgroup
Equation~\eqref{V_lyap_sylv2} gives
\begin{equation}
A^\top Z_{\mathrm{lyap}}^{(i-1)}=-E^\top Z_{\mathrm{lyap}}^{(i-1)} (S_{\mathrm{lyap}}^{(i-1)})^\top-(\mathcal{C}_{\perp}^{(i-1)})^\top L_{\mathrm{lyap}}^{(i-1)}.\label{obsv2}
\end{equation}
Moreover,
\begingroup\interdisplaylinepenalty=10000
\begin{align}
(A^\top+\alpha_i E^\top)\bar{z}_i&=\frac{1}{\gamma_i}(A^\top+\alpha_i E^\top)z_i=(\mathcal{C}_{\perp}^{(i-1)})^\top,\label{obsv3}\\
\mathcal{C}_{\perp}^{(i-1)}&=\hat{C}-L_{\mathrm{lyap}}^{(i-1)}(Z_{\mathrm{lyap}}^{(i-1)})^\top E.\label{obsv4}
\end{align}
\endgroup
Now consider the left-hand side of the second equality in the chain above:
\begingroup\interdisplaylinepenalty=10000
\begin{align}
\mathrm{LHS}&=\Big(\hat{A}^\top-E^\top Z_{\mathrm{lyap}}^{(i-1)}T_{\mathrm{radi}}^{(i-1)}X_r^{(i-1)}(W^{(i-1)})^\top\hat{B}\hat{R}^{-1}\hat{B}^\top+\alpha_iE^\top\Big)Z_{\mathrm{lyap}}^{(i)}t_{\mathrm{radi}}^{(i)}\nonumber\\
&=\Big(A^\top+\hat{C}^\top\hat{Z}\tilde{B}^\top-E^\top Z_{\mathrm{lyap}}^{(i-1)}T_{\mathrm{radi}}^{(i-1)}X_r^{(i-1)}\hat{B}_r^{(i-1)}\hat{R}^{-1}\hat{B}^\top+\alpha_iE^\top\Big)\Big(Z_{\mathrm{lyap}}^{(i-1)}t_1^{(i)}+z_it_{2}^{(i)}\Big)\nonumber\\
&=A^\top Z_{\mathrm{lyap}}^{(i-1)}t_1^{(i)}-E^\top Z_{\mathrm{lyap}}^{(i-1)}T_{\mathrm{radi}}^{(i-1)}X_r^{(i-1)}\hat{B}_r^{(i-1)}\hat{R}^{-1}\hat{B}^\top Z_{\mathrm{lyap}}^{(i-1)}t_1^{(i)}\nonumber\\
&\quad+\alpha_iE^\top Z_{\mathrm{lyap}}^{(i-1)}t_1^{(i)}+(A^\top+\alpha_iE^\top ) z_it_{2}^{(i)}-E^\top Z_{\mathrm{lyap}}^{(i-1)}T_{\mathrm{radi}}^{(i-1)}X_r^{(i-1)}\hat{B}_r^{(i-1)}\hat{R}^{-1}\hat{B}^\top z_it_{2}^{(i)}\nonumber\\
&\quad+\hat{C}^\top\hat{Z}\tilde{B}^\top Z_{\mathrm{lyap}}^{(i-1)}t_1^{(i)}+\hat{C}^\top\hat{Z}\tilde{B}^\top z_it_{2}^{(i)}.\nonumber
\end{align}
\endgroup
Substituting \eqref{obsv1}--\eqref{obsv4} into this expression yields
\begin{align}
\mathrm{LHS}&=E^\top Z_{\mathrm{lyap}}^{(i-1)}\Big(-(S_{\mathrm{lyap}}^{(i-1)})^\top +\alpha_i I\nonumber\\
&\qquad-T_{\mathrm{radi}}^{(i-1)}X_r^{(i-1)}\hat{B}_r^{(i-1)}\hat{R}^{-1}(\hat{B}_{\mathrm{lyap}}^{(i-1)})^\top \Big)t_1^{(i)}\nonumber\\
&\quad+(\mathcal{C}_{\perp}^{(i-1)})^\top\Big(\gamma_it_{2}^{(i)}-L_{\mathrm{lyap}}^{(i-1)}t_1^{(i)}\Big)\nonumber\\
&\quad-E^\top Z_{\mathrm{lyap}}^{(i-1)}T_{\mathrm{radi}}^{(i-1)}X_r^{(i-1)}\hat{B}_r^{(i-1)}\hat{R}^{-1}\hat{B}^\top z_it_{2}^{(i)}\nonumber\\
&\quad+\hat{C}^\top\hat{Z}\tilde{B}^\top Z_{\mathrm{lyap}}^{(i-1)}t_1^{(i)}+\hat{C}^\top\hat{Z}\tilde{B}^\top z_it_{2}^{(i)}\nonumber\\
&=E^\top Z_{\mathrm{lyap}}^{(i-1)}\Big(a_{11}t_1^{(i)}+a_{12}t_{2}^{(i)}-(L_{\mathrm{lyap}}^{(i-1)})^\top\hat{Z}(\bar{B}_{\mathrm{lyap}}^{(i-1)})^\top t_1^{(i)}\nonumber\\
&\qquad-(L_{\mathrm{lyap}}^{(i-1)})^\top\hat{Z}\tilde{B}^\top z_it_{2}^{(i)}\Big)\nonumber\\
&\quad+(\mathcal{C}_{\perp}^{(i-1)})^\top\big(\hat{Z}-\hat{Z}(\bar{B}_{\mathrm{lyap}}^{(i-1)})^\top t_1^{(i)}-\hat{Z}\tilde{B}^\top z_it_2^{(i)}\big)\nonumber\\
&\quad+\hat{C}^\top\hat{Z}\tilde{B}^\top Z_{\mathrm{lyap}}^{(i-1)}t_1^{(i)}+\hat{C}^\top\hat{Z}\tilde{B}^\top z_it_{2}^{(i)}\nonumber\\
&=E^\top Z_{\mathrm{lyap}}^{(i-1)}\Big((L_{\mathrm{lyap}}^{(i-1)})^\top\hat{Z}-T_{\mathrm{radi}}^{(i-1)}X_r^{(i-1)}(L_r^{(i-1)})^\top\hat{Z}\nonumber\\
&\qquad-(L_{\mathrm{lyap}}^{(i-1)})^\top\hat{Z}(\bar{B}_{\mathrm{lyap}}^{(i-1)})^\top t_1^{(i)}-(L_{\mathrm{lyap}}^{(i-1)})^\top\hat{Z}\tilde{B}^\top z_it_{2}^{(i)}\Big)\nonumber\\
&\quad+(\mathcal{C}_{\perp}^{(i-1)})^\top\big(\hat{Z}-\hat{Z}(\bar{B}_{\mathrm{lyap}}^{(i-1)})^\top t_1^{(i)}-\hat{Z}\tilde{B}^\top z_it_2^{(i)}\big)\nonumber\\
&\quad+\hat{C}^\top\hat{Z}\tilde{B}^\top Z_{\mathrm{lyap}}^{(i-1)}t_1^{(i)}+\hat{C}^\top\hat{Z}\tilde{B}^\top z_it_{2}^{(i)}\nonumber\\
&=E^\top Z_{\mathrm{lyap}}^{(i-1)}\Big((L_{\mathrm{lyap}}^{(i-1)})^\top\hat{Z}-T_{\mathrm{radi}}^{(i-1)}X_r^{(i-1)}(L_r^{(i-1)})^\top\hat{Z}\nonumber\\
&\qquad-(L_{\mathrm{lyap}}^{(i-1)})^\top\hat{Z}(\bar{B}_{\mathrm{lyap}}^{(i-1)})^\top t_1^{(i)}-(L_{\mathrm{lyap}}^{(i-1)})^\top\hat{Z}\tilde{B}^\top z_it_{2}^{(i)}\Big)\nonumber\\
&\quad+\hat{C}^\top\hat{Z}-\hat{C}^\top\hat{Z}(\bar{B}_{\mathrm{lyap}}^{(i-1)})^\top t_1^{(i)}-\hat{C}^\top\hat{Z}\tilde{B}^\top z_it_2^{(i)}\nonumber\\
&\quad-E^\top Z_{\mathrm{lyap}}^{(i-1)}\Big((L_{\mathrm{lyap}}^{(i-1)})^\top\hat{Z}-(L_{\mathrm{lyap}}^{(i-1)})^\top\hat{Z}(\bar{B}_{\mathrm{lyap}}^{(i-1)})^\top t_1^{(i)}\nonumber\\
&\qquad-(L_{\mathrm{lyap}}^{(i-1)})^\top\hat{Z}\tilde{B}^\top z_it_2^{(i)}\Big)\nonumber\\
&\quad+\hat{C}^\top\hat{Z}\tilde{B}^\top Z_{\mathrm{lyap}}^{(i-1)}t_1^{(i)}+\hat{C}^\top\hat{Z}\tilde{B}^\top z_it_{2}^{(i)}\nonumber\\
&=\hat{C}^\top\hat{Z}-E^\top Z_{\mathrm{lyap}}^{(i-1)}T_{\mathrm{radi}}^{(i-1)}X_r^{(i-1)}(L_r^{(i-1)})^\top\hat{Z}.\nonumber
\end{align}
For $i=1$, the matrices $Z_{\mathrm{lyap}}^{(i-1)}$, $T_{\mathrm{radi}}^{(i-1)}$, and $X_r^{(i-1)}$ are empty, so
\begingroup\interdisplaylinepenalty=10000
\begin{align}
\mathrm{LHS}&=(\hat{A}^\top+\alpha_1 E^\top)Z_{\mathrm{lyap}}^{(1)}t_{\mathrm{radi}}^{(1)}\nonumber\\
&=(A^\top+\hat{C}^\top\hat{Z}\tilde{B}^\top+\alpha_1 E^\top)z_1t_{\mathrm{radi}}^{(1)}\nonumber\\
&=\Big((A^\top+\alpha_1 E^\top)z_1+\hat{C}^\top\hat{Z}(\bar{B}_{\mathrm{lyap}}^{(1)})^\top\Big)t_{\mathrm{radi}}^{(1)}\nonumber\\
&=\big(\gamma_1\hat{C}^\top+\hat{C}^\top\hat{Z}(\bar{B}_{\mathrm{lyap}}^{(1)})^\top\big)t_{\mathrm{radi}}^{(1)}\nonumber\\
&=\hat{C}^\top\big(\gamma_1I+\hat{Z}(\bar{B}_{\mathrm{lyap}}^{(1)})^\top \big)t_{\mathrm{radi}}^{(1)}.\nonumber
\end{align}
\endgroup
Substituting
$t_{\mathrm{radi}}^{(1)}=\big(2\alpha_1 I-\gamma_1\hat{Z}(\bar{B}_{\mathrm{lyap}}^{(1)})^\top\big)^{-1}(-\gamma_1\hat{Z})=\big(\gamma_1I+\hat{Z}(\bar{B}_{\mathrm{lyap}}^{(1)})^\top\big)^{-1}\hat{Z}$
from \eqref{ti} gives
\[
\mathrm{LHS}=\hat{C}^\top\hat{Z}=(\hat{A}^\top+\alpha_1E^\top)w_1.
\]
Hence,
$(\hat{A}^\top+\alpha_1E^\top)w_1=(\hat{A}^\top+\alpha_1 E^\top)Z_{\mathrm{lyap}}^{(1)}t_{\mathrm{radi}}^{(1)}$,
and, since $\hat{A}^\top+\alpha_1E^\top$ is nonsingular, $w_1=Z_{\mathrm{lyap}}^{(1)}t_{\mathrm{radi}}^{(1)}$.

For $i=2$,
\begingroup\interdisplaylinepenalty=10000
\begin{align}
&\Big(\hat{A}^\top-E^\top W^{(1)}X_r^{(1)}(W^{(1)})^\top\hat{B}\hat{R}^{-1}\hat{B}^\top+\alpha_2E^\top\Big)Z_{\mathrm{lyap}}^{(2)}t_{\mathrm{radi}}^{(2)}\nonumber\\
&=\hat{C}^\top\hat{Z}-E^\top W^{(1)}X_r^{(1)}(L_r^{(1)})^\top\hat{Z}\nonumber\\
&=(C_{\perp}^{(1)})^\top \hat{Z}\nonumber\\
&=\Big(\hat{A}^\top-E^\top W^{(1)}X_r^{(1)}(W^{(1)})^\top\hat{B}\hat{R}^{-1}\hat{B}^\top+\alpha_2E^\top\Big)w_2,\nonumber
\end{align}
\endgroup
so $w_2=Z_{\mathrm{lyap}}^{(2)}t_{\mathrm{radi}}^{(2)}$, again by nonsingularity of the left-hand factor. The same argument applies recursively for $i=3,4,\ldots$.

Finally, observe that
\[
Z_{\mathrm{lyap}}^{(i)}T_{\mathrm{radi}}^{(i)}=\begin{bmatrix}Z_{\mathrm{lyap}}^{(i-1)}&z_i\end{bmatrix}\begin{bmatrix}T_{\mathrm{radi}}^{(i-1)}&t_1^{(i)}\\0&t_2^{(i)}\end{bmatrix}
=\begin{bmatrix}Z_{\mathrm{lyap}}^{(i-1)}T_{\mathrm{radi}}^{(i-1)}&Z_{\mathrm{lyap}}^{(i)}t_{\mathrm{radi}}^{(i)}\end{bmatrix}.
\]
Since $Z_{\mathrm{lyap}}^{(1)}T_{\mathrm{radi}}^{(1)}=W^{(1)}$, applying this identity recursively gives $W^{(i)}=Z_{\mathrm{lyap}}^{(i)}T_{\mathrm{radi}}^{(i)}$ for all $i$.
\end{proof}

\end{document}